\documentclass{article}
\usepackage{amsmath}
\usepackage{amsfonts}
\usepackage{amsthm}
\usepackage{centernot}
\usepackage{url}

\def\congruent{\equiv}
\def\mod{\bmod}

\def\notmid{\centernot\mid}
\def\ratQ{\mathbb{Q}}

\def\scripta{ \mathfrak{a} } 
\def\scriptp{ \mathfrak{p} } 

\newtheorem{Theorem}{Theorem}
\newtheorem{Lemma}{Lemma}

\newtheorem*{Corollary}{Corollary to Lemma~\ref{Lem:MF}}

\def\Jacobi#1#2{{\left( {#1 \over #2} \right)}}
\def\Legendre#1#2{\left( {#1 \over #2} \right)}

\begin{document}

\begin{center}
{\large\bf 
The equation $\vert p^x \pm q^y \vert = c$ in nonnegative $x$, $y$. 
}

\bigskip

Reese Scott

Robert Styer 
\end{center}



revised 16 Dec 2011  

\bigskip

\begin{abstract}  
We improve earlier work on the title equation (where $p$ and $q$ are primes and $c$ is a positive integer) by allowing $x$ and $y$ to be zero as well as positive.    Earlier work on the title equation showed that, with listed exceptions, there are at most two solutions in positive integers $x$ and $y$, using elementary methods.  Here we show that, with listed exceptions, there are at most two solutions in nonnegative integers $x$ and $y$, but the proofs are dependent on nonelementary work of Mignotte, Bennett, Luca, and Szalay.  In order to provide some of our results with purely elementary proofs, we give short elementary proofs of the results of Luca, made possible by an elementary lemma which also has an application to the familiar equation $x^2 + C = y^n$.  We also give shorter simpler proofs of Szalay's results.  A summary of results on the number of solutions to the generalized Pillai equation $(-1)^u r a^x + (-1)^v s b^y = c$ is also given.  
\end{abstract}   

MSCN:{ 11D61 }

\bigskip

\section{Introduction}  

Earlier work (\cite{Be}, \cite{Le}, \cite{Sc}, \cite{ScSt}, \cite{ScSt2}) has treated the equation 
$$ (-1)^u a^x + (-1)^v b^y = c  $$
for integers $a>1$, $b>1$, $c>0$, with solutions $(x,y, u, v)$ where $u,v \in \{0,1\}$ and $x$ and $y$ are positive integers.  Recently, in treating the more general Pillai equation 
$$(-1)^u r a^x + (-1)^v s b^y = c  \eqno{({\rm P})}  $$
(where $r$ and $s$ are positive integers), the authors noticed that it is in a sense a more natural approach to allow $x$ and $y$ to be zero as well as positive; this is because analyzing (P) is greatly clarified by the use of what the authors in \cite{ScSt4} call {\it basic forms}, which require exponents equal to zero (see Lemma 1 of \cite{ScSt4} for a definition of basic form).  

So in this paper we improve earlier results (Theorems 3 and 5 of \cite{Sc} and Theorem 7 of \cite{ScSt}) by allowing the variables $x$ and $y$ in $(-1)^u a^x + (-1)^v b^y = c$ to be zero as well as positive, which significantly alters the nature of the proofs: while the proofs in \cite{Sc} and \cite{ScSt} are elementary, the proofs of Theorems 1, 2, and 3 below depend on non-elementary work of Luca~\cite{Lu} and  Szalay~\cite{Sz}, and the proof of Theorem 3 depends also on non-elementary results of Mignotte~\cite{Mi} and Bennett~\cite{Be}.  We have not been able to remove dependence on these non-elementary results, but we have been able to replace the proofs in \cite{Lu} and some of the proofs in \cite{Sz} by short elementary proofs, thus making Theorems 1 and 2 elementary.  (For this reason we state Theorem 2 and the nonelementary Theorem 3 separately even though Theorem 3 includes Theorem 2 except for the trivial case $p=q$.)  For the most part, we restrict the bases $a$ and $b$ to prime values, noting that it seems likely that the list of exceptional cases in Theorem 3 would remain unchanged even if composite values were allowed (see the discussion in Section 2).  We prove the following results:

\begin{Theorem}   
\label{Thm1}
For integers $b>1$ and $c>0$ and positive prime $a$, the equation
\begin{equation}   a^x - b^y = c   \label{1}  \end{equation}
has at most two solutions in nonnegative integers $(x,y)$, except for $(a,b,c) = (2,5,3)$, which has solutions $(x,y)=(2,0)$, $(3,1)$, $(7,3)$.  

There are an infinite number of $(a,b,c)$ for which (\ref{1}) has two solutions.  
\end{Theorem}

\begin{Theorem}   
\label{Thm2} 
For positive primes $p$ and $q$ and positive integer $c$, the equation
\begin{equation} \vert p^x - q^y \vert = c \label{5} \end{equation}
has at most two solutions in nonnegative integers $x$ and $y$, except when $(p,q,c)$ or $(q,p,c) = (3,5,2)$, $(2,3,5)$, $(2,3,7)$, $(2, 11, 7)$, or $(2, F, F-2)$ where $F$ is a Fermat prime.  
\end{Theorem}

\begin{Theorem} 
\label{Thm3}
For distinct positive primes $p$ and $q$ and positive integer $c$ there are at most two solutions to the equation
\begin{equation}  (-1)^u  p^x + (-1)^v q^y = c   \label{1c} \end{equation}
in nonnegative integers $x$ and $y$ and integers $u, v \in \{0,1\}$,  
except when $(p,q,c)$ or $(q,p,c)$ is one of the following:
$(2,3,1)$, $(2,3,5)$, $(2,3,7)$, $(2,3,11)$, $(2,3,13)$, $(2,3,17)$, $(2,5,3)$, $(2,5,7)$, $(2,5,9)$, $(2,11,7)$, $(3,5,2)$, $(3,5,4)$, $(3,13,10)$, $(2, F, F-2)$, $(2, F, 2F-1)$, $(2, M, M+2)$, $(2, M, 2M+1)$, $(3, 3^n +(-1)^\delta 2, 2)$, $(2, 2^t +(-1)^\delta 3, 3)$ where $F>5$ is a Fermat prime, $M>3$ is a Mersenne prime, $\delta \in \{0, 1\}$, $n>1$ is a positive integer such that $(n, \delta) \ne (3,1)$, and $t>1$ is a positive integer such that $(t, \delta) \ne (2, 1)$, $(3, 1)$, or $(7,1)$.  

The solutions in these cases are as follows:
$$
\begin{matrix} 
2-1 = -2+3 = 2^2-3 = -2^3 + 3^2 = 1 \\
2^2 + 1 = 2 + 3 = 2^3 - 3 = -2^2 + 3^2 = 2^5 - 3^3 = 5 \\
2^3 - 1 = 2^2 + 3 = -2 + 3^2 = 2^4 - 3^2 = 7 \\
2^3 + 3 = 2 + 3^2 = -2^4 + 3^3 = 11 \\
2^2 + 3^2 = 2^4 - 3 = 2^8 - 3^5 = 13 \\
2^4 + 1 = 2^3 + 3^2 = -2^6 + 3^4 = 17 \\
2+1 = 2^2 - 1 = -2 + 5 = 2^3 - 5 = 2^7 - 5^3 = 3 \\
2^3 - 1 = 2+5= 2^5 - 5^2 = 7 \\
2^3 + 1 = 2^2 + 5 = -2^4 + 5^2 = 9 \\
2^3 - 1 = -2^2 + 11 = 2^7 - 11^2 = 7 \\
1+1 = 3-1 = -3 + 5 = 3^3 - 5^2 = 2 \\
3+1 = -1 + 5 = 3^2 - 5 = 4 \\
3^2 + 1 = -3 + 13 = -3^7 + 13^3 = 10 \\
(F-1) - 1 = -2 + F = (2F-2) - F = F-2 \\
(2F-2) + 1 = (F-1) + F = -(F-1)^2 + F^2 = 2F - 1 \\
(M+1) + 1 = 2 + M = (2M+2) - M = M + 2 \\
(2M+2) - 1 = (M+1) + M = (M+1)^2 - M^2 = 2M+1 \\
1 + 1 = 3 - 1 =  -(-1)^{\delta} 3^n + (-1)^\delta (3^n +(-1)^\delta 2) = 2 \\
2 + 1 = 2^2 - 1 = -(-1)^\delta 2^t + (-1)^\delta (2^t + (-1)^\delta 3) = 3 
\end{matrix}
$$
\end{Theorem}

We give the new elementary proofs of the results in \cite{Lu} and \cite{Sz} in Section 3.  The key to making these elementary proofs possible is the elementary proof of Lemma 2 in Section 3, which also has a further application which we give in Section 6:  we establish a bound on $n$ in the familiar equation $x^2 + C = y^n$, when $x$ and $y$ are primes or prime powers and $2 \mid C$.  The bound depends only on the primes dividing $C$ and the result is elementary.  Beukers~\cite{Bk2} established a bound on $n$ for more general $x$ when $y=2$, and Bauer and Bennett~\cite{BB} greatly improved this bound as well as allowing $y$ to take on many specific values.  The bounds of \cite{Bk2} and \cite{BB} depend on the value of $y$ and the specific value of $C$.  See also earlier results of Nagell~\cite{N} and Ljunggren~\cite{Lj}.   

Before proceeding, we give a brief discussion of these changes in the proofs in \cite{Lu} and \cite{Sz}, which deal with the equation 
$$p^r \pm p^s + 1 = z^2, $$
where $p$ is a prime and $z$, $r$, and $s$ are positive integers.   
Luca \cite{Lu} handles the case $p>2$ using lower bounds on linear forms in logarithms (see \cite[pp.~7--11]{Lu}) and the well known recent work of Bilu, Hanrot, and Voutier \cite{BHV} (see \cite[pp.~12--14]{Lu}).  In Section 3 we obtain short and elementary proofs of Luca's results, without interfering with the clever use of continued fractions in \cite[equation (18)]{Lu}, by using two elementary lemmas which replace the use of linear forms in logarithms and \cite{BHV} (see Lemmas 1 and 2 in Section 3).  Further, in proving Lemma 4 of this paper, we have removed Luca's use of work of Carmichael \cite{Car}.  Gary Walsh pointed out to the first author that \cite{Car} is not needed for proving an auxiliary lemma used by Luca to prove Lemma 5 of this paper; although this auxiliary lemma is not used in our proof of Lemma 5, Walsh's comment led to our new proof of Lemma 4.   
 
Szalay \cite{Sz} handles the equation $2^r - 2^s + 1 = z^2$ using a non-elementary bound of Beukers \cite{Bk2}.  However, an earlier result of Beukers, the elementary Theorem 4 of \cite{Bk1}, can be used instead, making Szalay's result elementary, so we will not need to give a new proof in this case.  Szalay \cite{Sz} also handles the case $2^r + 2^s + 1 = z^2$ using a nonelementary result in \cite{Bk2}.  In this case we have not obtained a strictly elementary proof; however, we do give a shorter proof of Szalay's result for the case $2^r + 2^s + 1 = z^2$ by replacing the older bound in \cite{Bk2} with the recent sharp result of Bauer and Bennett \cite{BB}, not available to Szalay.  Szalay's proof can be further shortened by observing that the methods of his Lemma 8 alone suffice to give the desired contradiction to Beukers' (or Bauer and Bennett's) results; the remaining auxiliary results in \cite{Sz}, including the mapping of one set of solutions onto another, are of independent interest.  An outline of a proof of this result was also given by Mignotte; see the comments at the end of Section D10 of \cite{G}.  

We are grateful to Michael Bennett for proving $y_3 = 1$ in equations (\ref{UV1}) and (\ref{UV2}) below by pointing out references \cite{BS} and \cite{BVY}.

\section{Context of the Problem} 

Before proceeding to the proofs, we view the results of this paper in the context of the following more general problem: 
for given integers $a>1$, $b>1$, $c>0$, $r>0$, and $s>0$, we consider $N$, the number of solutions $(x,y,u,v)$ to the generalized Pillai equation
$$  (-1)^u r a^x + (-1)^v s b^y = c \eqno{({\rm P})} $$
in nonnegative integers $x$, $y$ and integers $u$, $v \in \{0,1\}$.  Note that the choice of $x$ and $y$ uniquely determines the choice of $u$ and $v$, so we will usually refer to a solution $(x,y)$.   

\noindent 
{\it The Case $(ra,sb) = 1$ }

There are only a finite number of cases with $N > 3$ solutions to Equation (P) \cite{ScSt4}.  There are at least five infinite families of cases with $N=3$ solutions to (P), as well as a number of anomalous cases with $N=3$ (by \lq anomalous case' we mean a case not a member of a known infinite family).  Some of these anomalous cases are quite high, e.g., $(a,b,c,r,s) = (56744, 1477, 83810889, 1478, 56743)$,  \cite{ScSt4}.  We have not been able to give a complete finite list of such anomalous solutions, so the question arises:  what additional restrictions on the variables would make possible a proof which gives a complete list of anomalous solutions, thus improving the result to $N=2$ except for completely designated exceptions?  This question has been essentially answered with the additional restriction $x>0$ and $y>0$ (see \cite{ScSt3}, in which the problem is reduced to a finite search).  But even if only one of the exponents $(x_1, x_2, \dots, x_N, y_1, y_2, \dots, y_N)$ is equal to zero, the problem becomes more difficult:  even with the further restriction $rs=1$ the methods of \cite{Be} and \cite{ScSt2} do not suffice without additional heavy restrictions such as placing an upper bound on one of $b$, $c$, or $\min(x_1, x_2, \dots, x_N)$ (when $\min(y_1, y_2 , \dots, y_N) = 0$).  But if one adds the yet further restriction that $a$ and $b$ be prime, it is possible to give a complete list of infinite families and a complete list of anomalous solutions, thus obtaining $N=2$ with completely designated exceptions (Theorem 3 of this paper).   The restriction that $a$ and $b$ be prime is perhaps not as artificial as it may seem: computer searches in \cite{Be} and \cite{ScSt2} (supplemented with calculations on the second author's website) suggest that the list of exceptions in Theorem 3 would remain unchanged even if $p$ and $q$ were allowed to be any relatively prime integers (here of course we would be redefining $F$ and $M$ to allow composite Fermat and Mersenne numbers).   

\noindent 
{\it The General Case}  

In what follows we will refer to a {\it set of solutions} to (P) which we will write as
$$(a,b,c,r,s: x_1, y_1; x_2, y_2; \dots ; x_N, y_N)$$
and by which we mean the (unordered) set of ordered pairs $\{(x_1, y_1), (x_2, y_2), \dots, (x_N, y_N)\}$ where each pair $(x_i, y_i)$ gives a solution $(x,y)$ to (P) for given integers $a$, $b$, $c$, $r$, and $s$.  We say that two sets of solutions $(a,b,c,r,s: x_1, y_1; x_2, y_2; \dots ; x_N, y_N)$ and $(A,B,C,R,S: X_1, Y_1; X_2, Y_2; \dots; X_N, Y_N)$ belong to the same {\it family} if $a$ and $A$ are both powers of the same integer, $b$ and $B$ are both powers of the same integer, there exists a positive rational number $k$ such that $k c = C$, and for every $i$ there exists a $j$ such that $kr a^{x_i} = R A^{X_j}$ and $ks b^{y_i} = S B^{Y_j}$, $1 \le i, j \le N$.  

If $(ra,sb)=1$, then $k =1$ and there are only a finite number of sets of solutions in each family; therefore, when $(ra,sb)=1$, we often dispense with the notion of family and deal simply with sets of solutions.  

Equation (P) has been treated by many authors, usually under at least one of the following additional restrictions:

\begin{tabular}{c l}
{(A.)} & $\min(x,y) \ge 1$, \\
{(B.)} &  $\min(x,y) \ge 2$, \\
{(C.)} & $(u,v) = (0,1)$,  \\
{(D.)} &  $(u,v) \ne (0,0)$ \\
{(E.)} &  $\gcd(ra, sb)=1$,  \\
{(F.)} &  $r=s=1$,  \\
{(G.)} &  $a$ is prime,  \\
{(H.)} &  $a$ and $b$ are both prime,  \\
{(I.)} &  $a$ and $b$ are both greater than a fixed real number, \\
{(J.)} &  terms on the left side of (P) are large relative to $c$  \\ 
\end{tabular} 
 
For any combination of such restrictions, we consider the problem of finding a number $N_0$ such that there are an infinite number of (families of) sets of solutions for which $N = n$ for every $n$ less than or equal to $N_0$ but only a finite number of (families of) sets of solutions for which $N > N_0$.  We also consider the problem of finding a number $M$ such that no sets of solutions have $N > M$, while sets of solutions exist with $N = M$.  The following table summarizes some known results, giving, for a given set of restrictions, results on $N_0$ and $M$ along with citations of sources.  In the column headed \lq\lq Restrictions'' we use the letters given in the list above, also writing \lq\lq K'' to mean \lq\lq no restrictions except those given in (P).''

\begin{tabular}{c c c}
Restrictions & Results & Sources \\
\hline \\ 
A C J   & $M \le 9$ &  \cite{Sh}  \\
B C E I  &  $M \le 3$ & \cite{Le} \\
B C E F I & $M \le 2$ & \cite{Le}  \\
A C F G  & $M = 2$ & \cite{Sc} \\
A D F H  & $N_0 \le 2$, $M = 3$ & \cite{Sc} \\
A F H  & $N_0 \le 2$, $M=4$  & \cite{ScSt}  \\
A C F   & $M = 2$ & \cite{Be} \\
A F   & $N_0 \le 2$, $M = 4$ & \cite{ScSt2} \\
B C E I  &  $M \le 2$ & \cite{HT} \\
A C E  &  $M \le 3$ & \cite{HT} \\
C   & $N_0 = 3$, $M = 3$ & \cite{ScSt3} \\
A E  & $N_0 = 2$, $M \ge 4$ & \cite{ScSt3} \\
K  & $N_0 = 3$, $M = 5$ & \cite{ScSt4}, \cite{ScSt5}  \\
\end{tabular} 

Lower bounds on linear forms in logarithms are used for the proofs of all the results cited in the table above except for those in \cite{Sc} and \cite{ScSt}, which are strictly elementary.  
In this paper, we show that strictly elementary methods suffice to improve the results in \cite{Sc} by eliminating the restriction (A.) and by obtaining a definite value for $N_0$ in the case with the set of restrictions C, F, G.  We can also eliminate the restriction (A.) in improving the result from \cite{ScSt}, although here our methods are not strictly elementary.  

Yet stronger restrictions can give $N_0 = 1$:  see for example \cite[Theorems 1.3, 1.4, 1.5, and Proposition 2.1]{Be}, \cite[Theorems 2 and 6]{ScSt},  \cite[Theorems 6 and 7]{ScSt2}, and \cite[Theorem 3]{Te}.

\section{Preliminary Lemmas}  

\begin{Lemma}  
{\label{L1}}
Let $D$ be any squarefree integer, let $u$ be a positive integer, and let $S$ be the set of all numbers of the form $r + s \sqrt{D}$, where $r$ and $s$ are nonzero rational integers, $(r,sD)=1$, and $u | s$.  Let $p$ be any odd prime number, and let $t$ be the least positive integer such that $\pm p^t $ is expressible as the norm of a number in $S$, if such $t$ exists.  Then, if $\pm p^n$ is also so expressible, we must have $t | n$.  (Note the $\pm$ signs in the statement of this lemma are independent.)   
\end{Lemma}

Comment: We will use this lemma when $D>0$ to bypass the problem of units.  

\begin{proof}   
Assume that for some $p$ and $S$, there exists $t$ as defined in the statement of the lemma.  Then $p$ splits in $\ratQ(\sqrt{D})$; let $[p] = P P'$.  For each positive integer $k$ there exists an $\alpha$ in $S$ such that $P^{kt} = [\alpha]$.  Now suppose $\pm p^{kt+g}$ equals the norm of $\gamma$ in $S$ where $k$ and $g$ are positive integers with $g < t$.  Since $P^{kt+g} $ must be principal, $P^{g} = [\beta]$ for some irrational integer $\beta \in \ratQ(\sqrt{D})$.  Therefore, for some unit $\epsilon$, either $\gamma = \epsilon \alpha \beta$ or $\bar{\gamma} = \epsilon \alpha \beta$.  $\epsilon \alpha \beta$ has integer coefficients and the norm of $\alpha$ is odd, so $\epsilon \beta$ has integer coefficients.  Now $\alpha \in S$ and $\epsilon \alpha \beta \in S$, so that one can see that $\epsilon \beta \in S$, which is impossible by the definitions of $t$ and $g$.  
\end{proof}

\begin{Lemma}  
{\label{L2}}   
The equation   
\begin{equation} (1+ \sqrt{-D} )^r = m \pm \sqrt{-D} \label{3.8} \end{equation}
has no solutions with integer $r>1$ when $D$ is a positive integer congruent to 2 mod 4 and $m$ is any integer, except for $D=2$, $r=3$.    
  
Further, when $D$ congruent to 0 modulo 4 is a positive integer such that $1+D$ is prime or a prime power, (\ref{3.8}) has no solutions with integer $r>1$ except for $D=4$, $r=3$.    
\end{Lemma} 
  
\begin{proof}  
Assume (\ref{3.8}) has a solution with $r>1$ for some $m$ and $D$.  From Theorem 13 of \cite{BH}, we see that $r$ is a prime congruent to 3 mod 4 and there is at most one such $r$ for a given $D$.  Thus we obtain   
\begin{equation} (-1)^{{D+2 \over 2}} = r - {r \choose 3} D + {r \choose 5} D^2 - \dots - D^{{r-1 \over 2}} \label{3.9} \end{equation} 
If $r=3$, (\ref{3.9}) shows that $|D-3| = 1$, giving the two exceptional cases of the Lemma.  So from here on we assume $3 \notmid r$.    
  
We will use two congruences:   
$$  (-1)^{{D+2 \over 2}} \congruent \Legendre{ r}{ 3} 2^{r-1} \bmod D-3 \leqno{{\rm Congruence\ 1:}}$$  
$$  (-1)^{{D+2 \over 2}} \congruent 2^{r-1} \bmod D+1  \leqno{{\rm Congruence\ 2:}}$$  
Congruences 1 and 2 correspond to congruences (9e) and (9f) of Lemma 7 of \cite{BH} and can  be derived by considering the expansions of $(1+\sqrt{-3})^r$ and $(1+1)^r$ respectively.  Noting that $r-1 \congruent 2 \bmod 4$, from Congruence 1 we see that $D-3$ cannot be divisible both by a prime 3 mod 4 and a prime 5 mod 8.  So $D \congruent 2 \bmod 4$ implies $D \not\congruent 3 \bmod 5$.  Now let $D+1 = y$.  If $D \congruent 1 \bmod 5$, $y^r \congruent 3 \bmod 5$; since $m^2 + D = y^r$, $m^2  \congruent 2 \bmod 5$, impossible.  If $D \congruent 2 \bmod 5$, $y^r \congruent 2 \bmod 5$, so that 5 divides $m$.  But then we see from (\ref{3.8}) that  $5 | m$ implies $3 | r$, which we have excluded.  Now $y^r$ is congruent to $-y$ modulo   
$y^2 + 1$ so that $m^2$ is congruent $-2 y + 1$ modulo $y^2 + 1$.  So, using the Jacobi symbol, we must have   
$$ 1 = \Jacobi{-2y+1}{(y^2+1)/2} = \Jacobi{2y^2+2}{ 2y-1} = \Jacobi{y+2}{ 2y-1} = \Jacobi{-5}{ y+2}. $$  
If $D \congruent 2 \bmod{4}$, then $y \congruent 3 \bmod 4$ and the last Jacobi symbol in this sequence equals $\Jacobi{y+2}{ 5} = \Jacobi{D+3}{5}$, which has the value $-1$ when $D$ is congruent to 0 or 4 modulo 5.  Thus, when $D \congruent 2 \bmod 4$ and $r \ne 3$, we have shown that there are no values of $D$ modulo 5 that are possible.    
  
So we assume hereafter that $D \congruent 0 \bmod 4$.  Write $D+1 = p^n$ where $p$ is prime, and let $g$ be the least number such that $2^g \congruent -1 \bmod p$, noting Congruence 2.  We see that $g | r-1$ and also $g | p-1 | p^n - 1 = D$.  Now (\ref{3.9}) gives $-1 \congruent 1 \bmod g$ so that $g \le 2$.  Assume first that $n$ is odd.  Since $4 | D$, $p \congruent 1 \bmod 4$.  In this case, we must have $g=2$, $p=5$.  If $n$ is even, since we have $1+D = p^n$ and $m^2 + D = p^{rn}$,  we must have $2 p^{rn / 2} - 1 \le D = p^n - 1$, giving $r < 2$, impossible.  So we have $n$ odd, $p=5$.    
  
Since $n$ is odd, $D \congruent 4 \bmod 8$, and, since $r \choose 3$ is odd, (\ref{3.9}) gives $r \congruent 3 \bmod 8$.  Now assume $r \congruent 2 \bmod 3$ and let $y = 5^n = 1+D$.  Then $y^r \congruent y^2 \bmod y^3 - 1$, so that $m^2 \congruent y^2 -y + 1 \bmod y^2 + y + 1$, so that   
$$1 = \Jacobi{y^2 - y + 1}{ y^2 + y + 1} = \Jacobi{-2y}{ y^2 +y+1} = \Jacobi{-2}{ y^2 + y + 1}$$  
which is false since $y^2 + y+1 \congruent 7 \bmod 8$.  Thus we have $r \congruent 19 \bmod 24$ so that $y^r \congruent -y^7 \bmod y^{12} + 1$, so that $m^2 \congruent -y^7 - y + 1 \bmod {y^{12} + 1 \over 2}$.  Thus we have   
\begin{align*} 
 1 &= \Jacobi{-y^7 - y + 1}{ (y^{12}+1)/2} = \Jacobi{y^7 + y - 1}{ (y^{12} + 1)/2} = \Jacobi{ 2 (y^{12} + 1)}{ y^7 + y- 1} \cr
&= \Jacobi{y^{12} + 1}{ y^7 + y -1} = \Jacobi{ y^6 - y^5 - 1}{ y^7 + y - 1} 
= \Jacobi{ y^7 + y -1}{ y^6 - y^5 - 1} \cr
&= \Jacobi{ y^5 + 2 y }{ y^6 - y^5 - 1 } = \Jacobi{y^4 + 2}{ y^6 - y^5 - 1}  
= - \Jacobi{y^6 - y^5 - 1}{ y^4 + 2} \cr 
&=  \Jacobi{2 y^2 -2y+1}{ y^4 + 2} = \Jacobi{ y^4 + 2}{ 2 y^2 - 2y + 1} = \Jacobi{7}{ 2 y^2 - 2y + 1} \cr
&= \Jacobi{ 2 y^2 - 2y + 1}{ 7}  
\end{align*}  
which is possible only when $y$ is congruent to 1, 4, or 0 modulo 7.  This is impossible since   
$y$ is an odd power of 5.  This completes the proof of the lemma.    
\end{proof} 

An almost immediate consequence of Lemma~\ref{L2} is the following:

\begin{Lemma} {(\cite{Lu})} 
\label{L3}  
The only solutions to the equation  
$$ p^r - p^s + 1 = z^2 $$  
in positive integers $(z,p,r,s)$ with $r>s$ and $p$ an odd prime are   
$(z,p,r,s) = (5,3,3,1)$, $(11,5,3,1)$.    
\end{Lemma}

\begin{proof} 
As in \cite{Lu}, we write $p^s - 1 = D u^2$, $D$ and $u$ positive integers and $D$ squarefree.  Clearly, $p$ splits in $\ratQ(\sqrt{-D})$, and we can let $[p] = \pi_1 \pi_2$ be its factorization into ideals.  We can take  
$${\pi_1}^s  = [ 1 + u \sqrt{-D} ], {\pi_1}^r = [z \pm u \sqrt{-D} ].  $$
At this point we diverge from \cite{Lu}: clearly $s$ is the least possible value of $n$ such that $p^n = h^2 + k^2 u^2 D$ for some relatively prime nonzero integers $h$ and $k$, so we can apply Lemma~\ref{L1} to obtain $s | r$.  Thus,   
$$( 1 + u \sqrt{-D} )^{r/s} = ( z \pm u \sqrt{-D} ) \epsilon $$  
where $\epsilon$ is a unit in $\ratQ(\sqrt{-D})$.  If $D = 1$ or 3, we note $2 | u$ and $ 2 \notmid z$, so that we must have $\epsilon = \pm 1$.  Now Lemma~\ref{L3} follows from Lemma~\ref{L2}.    
\end{proof}

Lemma~\ref{L3} is the only result from \cite{Lu} which we will need to prove Theorems~\ref{Thm1} and \ref{Thm2}.  However, for Theorem 3 we will also need Lemmas 4 and 5 below, for which we again give short elementary proofs:

\begin{Lemma}{(\cite{Lu})}  
\label{L4} 
The equation   
\begin{equation}  z^2 = w^r + \varepsilon_1 w^s + \varepsilon_2, \qquad \varepsilon_1, \varepsilon_2 \in \{1, -1\}, \label{1.4} \end{equation}   
has no positive integer solutions $(z,w,r,s)$ with $r>s$, $r$ even, and $w>2$.    
\end{Lemma}

\begin{proof}  
First we consider the case $s$ even.    
We establish some notation as in \cite{Lu}.  Letting $X=z$, $Y=w^{s/2}$, and $D = w^{r-s}+\varepsilon_1$, we rewrite (\ref{1.4}) as   
\begin{equation}X^2 - D Y^2 = \varepsilon_2. \label{(3.1)}\end{equation}  
The least solution of $U^2 - D V^2 = \pm 1$ is $(U, V) = (w^{(r-s)/2}, 1)$.  Write $X_n + Y_n \sqrt{D} = (w^{(r-s)/2} + \sqrt{D})^n$ for any integer $n$.  For some $j > 1$, $(X, Y) = (X_j, Y_j)$.  As in \cite{Lu}, it is easily seen that $2 | j$.  At this point we diverge from \cite{Lu} and apply Lemmas 1--3 of \cite{Sc} to see that, if $j > 2$, there exists a prime $q$ such that $q | w$, $q | (Y_j/Y_2)$, $Y_{2q} | Y_j$, and $Y_{2q}/ (q Y_2)$ is an integer prime to $w$.  But since $Y_{2q} / (q Y_2)$ is greater than 1 and divides $Y_j$, we have a contradiction.  So $j = 2$ and we must have 
\begin{equation}w^{s/2} = Y = Y_2 = 2 w^{(r-s)/2}. \label{(3.2)}\end{equation} 
  
Now we consider the case $s$ odd and again establish notation as in \cite{Lu}.  Letting $X=z$, $Y = w^{(s-1)/2}$, and $D = w ( w^{r-s} + \varepsilon_1)$, we rewrite (\ref{1.4}) as (\ref{(3.1)}).  At this point we diverge from \cite{Lu} and apply an old theorem of St\"ormer \cite{Sto}: his Theorem 1 says if every prime divisor of $Y$ divides $D$ in (\ref{(3.1)}), then $(X, Y) = (X_1, Y_1)$, the least solution of (\ref{(3.1)}).  Theorem 1 of \cite{Sto} also applies to show that $(2 w^{r-s} + \varepsilon_1,  2 w^{(r-s - 1)/2} )$ is the least solution $(U_1, V_1)$ of $U^2 - D V^2 = 1$.  If $\varepsilon_2 = -1$, then $2 X_1 Y_1 = 2 w^{(r-s-1)/2}$, which is impossible since $(X_1, w) = 1$, and $w>2$ implies $z = X_1 > 1$.  Thus we must have $\varepsilon_2 = 1$, so that 
\begin{equation}w^{(s-1)/2} = Y = Y_1 = V_1 = 2 w^{(r-s - 1)/2}. \label{(3.3)}\end{equation}
At this point we return to \cite{Lu} where it is pointed out that (\ref{(3.2)}) and (\ref{(3.3)}) require $w=2$ which is not under consideration.  
\end{proof} 

We note that Theorem 1 of \cite{Sto} has a short elementary proof.

\begin{Lemma}{(\cite{Lu})} 
\label{L5} 
There are no solutions to the equation  
\begin{equation}p^r + p^s + 1 = z^2 \label{1.5} \end{equation}  
in positive integers $(z,p,r,s)$ with $p$ an odd prime.     
\end{Lemma}

\begin{proof}
We first establish some notation by paraphrasing \cite[Section 3]{Lu}:   
Looking at (\ref{1.5}), we see that the only case in which solutions might exist is when $p \congruent 3 \bmod 4$ and $r-s$ is odd; choose $r$ odd and let $p^s + 1 = D u^2$, with $D$ square-free and $u > 0$ an integer.  At this point we diverge from \cite{Lu} and note that if $S$ is the set of all integers of the form $h + k \sqrt{D}$ with nonzero rational integers $h$ and $k$, $(h,kD) = 1$ and $u | k$, then $p^r$ and $-p^s$ are both expressible as the norms of numbers in $S$.  Therefore Lemma~\ref{L1} shows that $\pm p^d$ is expressible as the norm of a number in $S$, where $d$ divides both $r$ and $s$.  From this  point on, we return to the method of proof of \cite{Lu}: $r$ is odd and $s$ is even, so we have $d \le s/2$.  For some coprime positive integers $X$ and $Y$ such that $(X, p^s +1) = 1$, we must have  
\begin{equation}  X^2 - Y^2 (p^s+1) = \pm p^d. \label{(3.4)}\end{equation}  
(\ref{(3.4)}) corresponds to (17) in \cite{Lu}.  Since $| p^d| < \sqrt{p^s+1}$, $X/Y$ must be a convergent of the continued fraction for $\sqrt{p^s+1}$.  But then, since $p^s + 1$ is of the form $m^2 + 1$, we must have $p^d = \pm 1$, impossible.  
\end{proof}    
  
It has already been pointed out in the Introduction that the following lemma can be made elementary simply by replacing the result from \cite{Bk2} used in Szalay's proof by the elementary result \cite[Theorem 4]{Bk1}. 

\begin{Lemma}{(\cite{Sz})}  
\label{L6}
The equation   
$$  2^r - 2^s + 1 = z^2  $$  
has no solutions in positive integers $(r,s,z)$ with $r>s $ except for the following cases:  
$$ (r,s,z) = (2t, t+1, 2^t-1) { \rm \ for\ positive\ integer\ } t>1 $$  
$$  (r,s,z) = ( 5,3,5) $$  
$$  (r,s,z) = (7,3,11)  $$  
$$  (r,s,z) = (15,3,181)  $$  
\end{Lemma}

Lemma 6 is the only result from \cite{Sz} which we will need for Theorems 1 and 2.  For Theorem 3 we will use a further result from \cite{Sz} for which we have not found a purely elementary proof.  However, we do give a shorter simpler proof:  

\begin{Lemma}{(\cite{Sz})}  
\label{L7}
The equation   
\begin{equation}  2^r + 2^s + 1 = z^2  \label{1.1} \end{equation}  
has no solutions in positive integers $(r,s,z)$ with $r \ge s $ except for the following cases:  
\begin{equation} (r,s,z) = (2t, t+1, 2^t+1) { \rm \  for\ positive\ integer\ } t  \label{A} \end{equation}  
\begin{equation}  (r,s,z) = ( 5,4,7) \label{B} \end{equation}  
\begin{equation}  (r,s,z) = (9,4,23)  \label{C} \end{equation}  
\end{Lemma}

\begin{proof}  
Assume (\ref{1.1}) has a solution that is not one of (\ref{A}), (\ref{B}), or (\ref{C}).  It is an easy elementary result that the only solution to (\ref{1.1}) with $r=s$ is given by Case (\ref{A}) with $t = 1$, so we can assume hereafter $r>s$.  

Considering (\ref{1.1}) modulo 8, we get $s > 2$.  If $s=3$, then $2^r = z^2 - 2^3 - 1 = (z+3) (z-3)$, giving $z = 5$, which is Case (\ref{A}) with $t=2$, so we can assume hereafter $s>3$.   
  
Write $z = 2^t k \pm 1$ for $k$ odd and the sign chosen to maximize $t>1$.  In what follows, we will always take the upper sign when $z \congruent 1 \bmod{4}$ and the lower sign when $z \congruent 3 \bmod{4}$.      
  
We have   
\begin{equation}  2^r + 2^s + 1 = 2^{2t} k^2 \pm (k \mp 1) 2^{t+1} + 2^{t+1} + 1.   \label{(2.1)}  \end{equation}  
From this we see $s=t+1$ so that $t \ge 3$.  Now (\ref{(2.1)}) yields $r \ge 2t-1$ with equality only when $t=3$, $k=1$, and $z \congruent 3 \bmod 4$, which is Case (\ref{B}), already excluded.  So $r \ge 2t$, hence $r > 2t$ since Case (\ref{A}) has been excluded.  So now   
$$ k \mp 1 = 2^{t-1} g  {\rm \  for\ some\ odd\ } g>0.  $$  
We have  
\begin{equation}  2^{r-2t} = k^2 \pm g = 2^{2t-2} g^2  \pm 2^t g + 1 \pm g.  \label{(2.2)}\end{equation} 
(\ref{(2.2)}) yields $r-2t \ge 2t - 3$ with equality only when $t=3$, $g=1$, and $z \congruent 3 \bmod 4$, which is Case (\ref{C}), already excluded.  So now $g \pm 1 = 2^t h$ for some odd $h>0$. So we must have $g \ge 2^t \mp 1$.  Assume $z \congruent 3 \bmod{4}$.  Then from (\ref{(2.2)}) we derive   
\begin{equation}  2^{r-2t} > g^2 (2^{2t-2} -1) >  2^{2t} 2^{2t-3} = 2^{4t -3}.  \label{(2.3)}\end{equation}  
Now assume $z \congruent 1 \bmod{4}$.  Then  
$$ 2^{r-2t} > 2^{2t-2} g^2 \ge 2^{2t-2} ( 2^{2t} - 2^{t+1} + 1) > 2^{2t-2} 2^{2t-1} = 2^{4t-3}.  $$  
In both cases we have   
\begin{equation} r \ge 6t - 2 = 6 s - 8.  \label{(2.4)}\end{equation}    
  
Now we can use Corollary 1.7 in Bauer and Bennett \cite{BB}:   
$$r < {2 \over 2 - 1.48} { \log(2^s+1) \over \log(2) }. $$  
Thus,   
$$ r < {1 \over 0.26} { \log(2^s+1) \over \log(2^s) } s  < { 1 \over 0.26} {\log(17) \over \log(16) }  s  < 4s. $$   
Combining this with (\ref{(2.4)}) we obtain $s < 4$ which is impossible since $s>3$.  
\end{proof}

\section{Proofs of Theorems 1 and 2}  

Write $v_a(b)$ to mean the highest power of $a$ dividing $b$ for positive prime $a$ and nonzero integer $b$; thus, $a^{v_a(b)} || b$.

\begin{proof}[Proof of Theorem 1:]
If $a \mid b$, then, in any solution of (\ref{1}), $v_a(b^y) = v_a(c)$, so that (\ref{1}) cannot have two solutions $(x,y)$.  So we assume from here on that $(a,b)=1$.  

Clearly (\ref{1}) has at most one solution with $y=0$.  Applying Theorem 3 of \cite{Sc} and noting that none of the five exceptional cases of Theorem 3 of \cite{Sc} has a further solution with $2 \mid y > 0$ (see, for example, the proof of Theorem 5 of \cite{Sc}), we see that, if (\ref{1}) has more than two solutions in nonnegative integers $x$ and $y$, we must have exactly one solution with $y=0$ and exactly two further solutions.  
If these two further solutions are among the exceptional cases of Theorem 3 of \cite{Sc}, a solution with $y=0$ occurs only when $(a,b,c) = (2,5,3)$.  So from here on we exclude the five exceptional cases of Theorem 3 of \cite{Sc} and assume that we have three solutions $(x_1, y_1)$, $(x_2, y_2)$, $(x_3, y_3)$ with $y_1=0$ and $2 \notmid y_2-y_3$.  Without loss of generality, assume $2 \mid y_2= 2t$ for some integer $t$.  Then we have a solution to the equation 
\begin{equation}  b^{2t} + c = a^{x_2}, \label{2} \end{equation} 
as well as a solution to the equation 
\begin{equation} 1 + c = a^{x_1}. \label{3} \end{equation}   

Applying Theorem 4 of \cite{Sc} to the solutions $(x_2, y_2)$ and $(x_3, y_3)$ and noting that all the cases listed in (22) of \cite{Sc} have already been excluded, we see that $a$ must be odd.  
Combining (\ref{2}) and (\ref{3}), we get 
$$ a^{x_2} - a^{x_1} + 1 = b^{2t}, $$
contradicting Lemma~\ref{L3} unless $(a,b,c) = (3, 5, 2)$ or $(5, 11, 4)$.  Considering each of these two cases modulo 3, we see that neither case allows 
a solution to (\ref{1}) with $y$ odd, so by Theorem 3 of \cite{Sc} neither case has a third solution.  

It remains to show that there are an infinite number of $(a,b,c)$ for which (\ref{1}) has two solutions by noting that, for a given choice of $a$, $x_1$, $x_2$, we simply let $b= a^{x_2} - a^{x_1} + 1$ and $c= a^{x_1} -1$.      
\end{proof}

\begin{proof}[Proof of Theorem 2:]
Clearly three solutions are impossible if $p=q$, so we can assume $p$ and $q$ are distinct primes.  
Excluding the exceptions listed in the theorem, assume we have more than two solutions to (\ref{5}).  Clearly there is at most one solution for which $\min(x,y)=0$.  Noting that the exceptional cases of Theorem 5 of \cite{Sc} have been excluded, we can assume we have exactly one solution in which $\min(x,y)=0$ and exactly two further solutions.   After Theorem~\ref{Thm1} above, we see that, without loss of generality, it suffices to consider just two cases.

{\it Case 1:}  Assume (\ref{5}) has exactly three solutions in the following form:  
\begin{equation}    q^{y_1} + c = p^{x_1},   \label{6}  \end{equation}
\begin{equation}    p^{x_2} + c = q^{y_2},   \label{7}  \end{equation}  
\begin{equation}    1 + c = p^{x_3},  \label{8}  \end{equation}
where $x_i > 0$ and $y_j > 0$ for $1 \le i \le 3$, $1 \le j \le 2$.  

Consideration modulo 2 gives $q>2$.  Assume first also $p>2$.  Substituting (\ref{8}) into (\ref{6}) and (\ref{7})  we get $q^{y_1} \congruent 1 \bmod p$ and $q^{y_2} \congruent -1 \bmod p$, so $2 \mid y_1 = 2k$ for some positive integer $k$.  But then 
\begin{equation*} q^{2k} = p^{x_1} - c = p^{x_1} - p^{x_3} + 1, \end{equation*}
contradicting Lemma~\ref{L3} unless $(p,q,c)=(3,5,2)$ or $(5,11,4)$.  The case $(3,5,2)$ has been excluded and the case $(5,11,4)$ makes (\ref{7}) impossible modulo 11.  

So we can assume $p=2$.  
If $x_3 = 1$, then $c=1$, and it is a familiar elementary result that we must have $q=3$, giving an excluded case.  
So we can assume $x_3 \ge 2$ and also $x_1 \ge 3$.  

If $x_2 \ge 2$, then, substituting (\ref{8}) into (\ref{6}) and (\ref{7}) we get $q^{y_1} \congruent 1 \bmod 4$ and $q^{y_2} \congruent 3 \bmod 4$, so that $2 \mid y_1$, violating Lemma~\ref{L6} unless $q^{y_1/2} = 2^{x_3 - 1} -1$ for $x_3 > 3$, or $c=7$ with $q=3$, 5, 11, or 181.  Since we have $q \congruent 3 \bmod 4$ and have excluded the cases $(p,q,c) = (2, 3, 7)$ and $(2,11,7)$, we are left with $x_3 > 3$, $y_1 =2$, and $q = 2^{x_3 - 1} - 1$ (noting $q^{y_1/2} $ cannot be a perfect power).  
In this case, $\Legendre{c}{q} = \Legendre{2q+1}{q} = 1$, making (\ref{7}) impossible since also $\Legendre{2}{q} = 1$ and $q \congruent 3 \bmod 4$.  

It remains to consider $x_2 = 1$, in which case $q^{y_2} = 2^{x_3} + 1$.  $y_2 > 1$ requires $q^{y_2} = 9$, giving $(p,q,c)=(2,3,7)$ which has already been excluded.  So $q^{y_2} = q = F$, a Fermat prime, giving the final exceptional case in the formulation of Theorem 2 (note the case $x_3 = 1$ has already been dealt with).   This completes the proof of Case 1.  

{\it Case 2:}  Assume (\ref{5}) has exactly three solutions in the following form:  
\begin{equation}    p^{x_1} + c = q^{y_1},   \label{9}  \end{equation}
\begin{equation}    p^{x_2} + c = q^{y_2},   \label{10}  \end{equation}  
\begin{equation}    1 + c = p^{x_3},  \label{11}  \end{equation}
where $x_i > 0$ and $y_j > 0$ for $1 \le i \le 3$, $1 \le j \le 2$.

By Theorem 3 of \cite{Sc} we have $ 2 \notmid x_1 - x_2$, noting that the exceptional cases of Theorem 3 of \cite{Sc} for which $c \le 5$ have been excluded, while the exceptional cases of Theorem 3 of \cite{Sc} for which $c > 5$ do not allow a solution with $\min(x,y) = 0$.  Consideration modulo 2 gives $q>2$. 

Assume first $p>2$.  
Substituting (\ref{11}) into (\ref{9}) and (\ref{10}) we find $q^{y_1}\congruent q^{y_2} \congruent -1 \bmod p$, so that 
$v_2(y_1) = v_2(y_2)$.  So $v_2(q^{y_1}-1) = v_2(q^{y_2} - 1)$.  Now rewrite (\ref{9}) and (\ref{10}) as 
\begin{equation}  ( p^{x_1} - 1) + c = (q^{y_1} - 1), \label{9a} \end{equation} 
\begin{equation}  ( p^{x_2} - 1) + c = (q^{y_2} - 1). \label{10a} \end{equation} 
If $v_2(c) = v_2( q^{y_1} - 1) = v_2(q^{y_2} - 1)$, then $v_2( p^{x_1} - 1) > v_2(c)$ and $v_2( p^{x_2} - 1) > v_2(c)$.  But then, since at least one of $x_1$ and $x_2$ is odd, we get $v_2( p-1) > v_2( c)$, contradicting (\ref{11}).  On the other hand, if $v_2(c) \ne v_2( q^{y_1} -1)$, then we must have $v_2(p^{x_1} - 1) = v_2(p^{x_2} -1)$, violating $ 2 \notmid x_1 - x_2$.   

So we must have $p=2$.  Recalling $2 \notmid x_1 - x_2$, take $2 \notmid x_1$, $2 \mid x_2$.  Consideration modulo 3 gives $q \congruent 2 \bmod 3$, $2 \notmid y_1$, $2 \mid y_2$.  Now (\ref{10}) give $c \congruent 1 \bmod 4$, so that (\ref{11}) gives $c=1$, and it is a familiar elementary result that we must have $q=3$, giving an excluded case.  
\end{proof}

\section{Proof of Theorem 3}  

\bigskip

We will use the following lemma based on a result of Mignotte \cite{Mi} as used by Bennett \cite{Be}.  

\begin{Lemma} \label{Lem:M} 
Let $a>1$, $b>1$, $c>1$, $x>0$, and $y>0$ be integers such that $(a,b) = 1$ and 
$$ a^x - b^y = c. $$ 
Let $G = y / \log(a)$.  Then either 
\begin{equation} G < 2409.08 \label{1m} \end{equation}
or 
\begin{equation}  G < \frac{2 \log(c)} {\log(a) \log(b)} + 22.997 ( \log(G) + 2.405)^2.  \label{2m} \end{equation}
Also when $G = x / \log(b)$ we have  (\ref{1m}) or (\ref{2m}).  
\end{Lemma}

\begin{proof}  
When $G= x/ \log(b)$ the lemma can be derived in essentially the same way as Equation (11) of \cite{ScSt2}.  
Now assume both (\ref{1m}) and (\ref{2m}) fail to hold for $G = y / \log(a)$, so that (\ref{1m}) fails to hold for $G = x / \log(b)$.  But if (\ref{2m}) fails to hold for $G = y / \log(a) \ge 2409.08$, it must also fail to hold for any $G > y / \log(a)$, so that (\ref{2m}) fails to hold for $G = x / \log(b)$, a contradiction since we have shown at least one of (\ref{1m}) or (\ref{2m}) must hold for $G = x / \log(b)$.     
\end{proof}   

\begin{proof}[Proof of Theorem 3:]  
We will first show that the exceptional $(p,q,c)$ listed in the formulation of Theorem 3 are the only $(p,q,c)$ which could have three or more solutions to (\ref{1c}); then, at the end of the proof, we will find all solutions $(x,y)$ for these $(p,q,c)$. 

The exceptional cases of Theorems 3 and 4 of \cite{Sc}, Theorem 7 of \cite{ScSt}, and Theorem 2 of the present paper are all included in the list of exceptions of the formulation of Theorem 3 above.  So in what follows we will use all these results without explicitly dealing with the exceptional $(p,q,c)$.  

Note that (\ref{1c}) can have at most two solutions with $\min(x,y)=0$.  

We first handle the cases $(p,q) = (2,3)$, $(3,2)$, $(2,5)$, and $(5,2)$.  If one of these cases gives three solutions to (\ref{1c}), then $c$ is odd and there is at most one solution with $\min(x,y) = 0$, unless $c=3$ which gives the excluded case $(p,q,c) = (2,5,3)$ listed in the formulation of the theorem.  
So when (\ref{1c}) has more than two solutions with $\min(p,q)=2$ and $\max(p,q) \in \{ 3,5 \}$, we can assume we have at least two solutions for which $\min(x,y)>0$.  
Now Theorem 4 of \cite{Sc} and Pillai's results in \cite{Pi} suffice to give all $(p,q,c)$ such that $(p,q)=(2,3)$ or $(3,2)$ and (\ref{1c}) has at least two solutions for which $\min(x,y)>0$, and it is easily determined which of these $(p,q,c)$ give more than two solutions to (\ref{1c}) in nonnegative integers $x$ and $y$; we list such $(p,q,c)$ in the formulation of Theorem 3.  The methods of Pillai \cite{Pi} can be used in just the same way to handle the case $(p,q)=(2,5)$ or $(5,2)$, so that, again using also Theorem 4 of \cite{Sc}, we can list all $(p,q,c)$ such that $(p,q)=(2,5)$ or $(5,2)$ and (\ref{1c}) has more than two solutions.  So from here on we will assume
\begin{equation}  p=2 \implies q > 5, q=2 \implies p>5. \label{I} \end{equation} 

Also, in the following search for $(p,q,c)$ allowing three or more solutions to (\ref{1c}), we will exclude all the exceptional cases listed in Theorem 3 from consideration.  

After Theorem 7 of \cite{ScSt} and Theorem 2 of the present paper it suffices to consider only cases in which (\ref{1c}) has three solutions at least one of which has $\min(x,y)=0$ and at least one of which has $(u,v) = (0,0)$.  We divide the proof into thirteen such cases which can be seen to include all possibilities.  In each of these cases, $p \ge 2$ and $q \ge 2$ are distinct primes unless otherwise indicated (in the first three cases we specify $\min(p,q)>2$).    In the first nine cases, we assume exactly one of the exponents $\{x_1,x_2, x_3, y_1, y_2, y_3\}$ is zero and the rest are positive.  In the final four cases, more than one of the exponents is zero.  

Note: in all thirteen cases the explicitly written exponents $x_i$ and $y_j$ are assumed to be greater than zero ($1 \le i \le 3$, $1 \le j \le 3$).  Terms with exponent zero are written simply as \lq\lq 1''.

\bigskip\bigskip 

\noindent {\it Case 1}
\begin{equation}    1 + c = q^{y_1}    \label{A1} \end{equation} 
\begin{equation}    p^{x_2} + q^{y_2} = c   \label{A2} \end{equation}
\begin{equation}    q^{y_3} + c = p^{x_3}  \label{A3} \end{equation}
where $p$ and $q$ are odd primes.  
Substituting (\ref{A1}) into (\ref{A2}) and (\ref{A3}), we find $q \mid p^{x_2} + 1$ and $q \mid p^{x_3} + 1$, so that $v_2(x_2) = v_2(x_3)$, giving $p^{x_2} \congruent p^{x_3} \bmod 4$.  So
$$ q^{y_3} = p^{x_3} - c \congruent p^{x_2} - c = - q^{y_2} \bmod 4, $$
so 
\begin{equation} q \congruent 3 \bmod 4, 2 \notmid y_3 - y_2.  \label{A4} \end{equation}   
From (\ref{A1}) we have $\Legendre{c}{q} = -1$, so that, from (\ref{A2}) and (\ref{A3}), 
\begin{equation} \Legendre{p}{q} = -1.  \label{A5} \end{equation}
If $p \congruent 3 \bmod 4$, then (\ref{A5}) requires $\Legendre{q}{p} = 1$ so (\ref{A3}) requires $\Legendre{c}{p} = -1$ while (\ref{A2}) requires $\Legendre{c}{p} = 1$, a contradiction.
If $p \congruent 1 \bmod 4$, then (\ref{A5}) requires $\Legendre{q}{p} = -1$, while (\ref{A2}) and (\ref{A3}) require $\Legendre{c}{p} = \Legendre{q^{y_2}}{p} = \Legendre{q^{y_3}}{p}$, 
so that $2 \mid y_3 - y_2$, contradicting (\ref{A4}). 

\bigskip \bigskip

\noindent {\it Case 2}
\begin{equation}    1+ q^{y_1} = c   \label{B1} \end{equation}
\begin{equation}    p^{x_2} + q^{y_2} = c   \label{B2} \end{equation}
\begin{equation}    q^{y_3} + c = p^{x_3}    \label{B3} \end{equation} 
where $p$ and $q$ are odd primes.  
Substituting (\ref{B1}) into (\ref{B2}) and (\ref{B3}), we find that, by Lemma 3, $x_2$ is odd unless $(p,q,c)= (5,3,28)$ or $(11,5,126)$, and, by Lemma 5, $x_3$ is odd, making $(p,q,c)=(5,3, 28)$ impossible modulo 3; also (\ref{B3}) is impossible modulo 11 if $(p,q,c)=(11,5,126)$.  So we can assume $x_2$ and $x_3$ are both odd.  Rewrite (\ref{B2}) and (\ref{B3}) as 
\begin{equation} (p^{x_2} - 1) + (q^{y_2} + 1) =  c  \label{B2a} \end{equation} 
and 
\begin{equation}  (q^{y_3} - 1) + c = p^{x_3} - 1.  \label{B3a}   \end{equation}
Since $x_2$ and $x_3$ are both odd, $v_2(p^{x_2} - 1) = v_2(p^{x_3} - 1)$.  Suppose $v_2(p^{x_2} - 1) < v_2(c)$.  Then we must have, from (\ref{B2a}) and (\ref{B3a}), $v_2(q^{y_2} + 1) = v_2(p^{x_2} -1) = v_2(p^{x_3} -1) = v_2(q^{y_3} - 1)$; this is possible only if $ q \congruent 3 \bmod 4$ and $v_2(q^{y_2} + 1) = v_2(q^{y_3} - 1) = 1$ so we must have $v_2(p^{x_2} - 1) = v_2(p^{x_3} -1) = 1$.  So now write equations (\ref{B2}) and (\ref{B3}) as
\begin{equation} (p^{x_2} + 1) + (q^{y_2} -1) = c \label{B2b} \end{equation} 
and
 \begin{equation}  (q^{y_3} + 1) + c = p^{x_3} + 1. \label{B3b} \end{equation} 
Note that in both (\ref{B2b}) and (\ref{B3b}) all three terms have valuation base 2 greater than 1 when $v_2(p^{x_2} - 1) < v_2(c)$.  Therefore, $y_1$ and $y_3$ are both odd so that $v_2(c) = v_2(q^{y_3} + 1)$.  Therefore, from (\ref{B3b}), we have $v_2(p^{x_3} + 1) > v_2(c)$ and since $v_2(p^{x_3} + 1) = v_2(p^{x_2} + 1)$, we  have  $v_2(p^{x_2} + 1) > v_2(c)$.  But we must also have $y_2$ even and $y_1$ odd so that $v_2(q^{y_2} -1) > v_2(c)$.  Thus (\ref{B2b}) becomes impossible, eliminating the possibility $v_2(p^{x_2} -1) < v_2(c)$. 

Now suppose $v_2(c) < v_2(p^{x_2} -1) = v_2(p^{x_3} -1)$.  Now from (\ref{B2a}) and (\ref{B3a}) we see that $v_2(c) = v_2(q^{y_2} + 1) = v_2(q^{y_3} - 1) = 1$.  
Now write (\ref{B2}) and (\ref{B3}) as  
\begin{equation} (p^{x_2} -1) + (q^{y_2} -1) = (c - 2)     \label{B2c}  \end{equation} 
and 
\begin{equation} (q^{y_3} +1) + (c-2) = p^{x_3} - 1.     \label{B3c}   \end{equation} 
Note that in both (\ref{B2c}) and (\ref{B3c}) all three terms have a valuation base 2 greater than 1.  We must have $q \congruent 3 \bmod 4$ with $v_2(q^{y_3} + 1) < v_2(q^{y_1} -1) = v_2(c-2)$, so that, from (\ref{B3c}), $v_2(c-2) > v_2(p^{x_3} -1) = v_2(p^{x_2} - 1)$, so that $v_2(q^{y_3} + 1) = v_2(p^{x_3} - 1) = v_2(p^{x_2} - 1)= v_2(q^{y_2} -1)$, which is impossible.  This eliminates the possibility $v_2(c) < v_2(p^{x_2} -1)$.  

So we are left with $v_2(c) = v_2(p^{x_2} - 1) = v_2(p^{x_3} -1)$.  
In this case from (\ref{B2a}) we see that $v_2(q^{y_2} + 1) > v_2(c)$ so that $q \congruent 3 \bmod 4$ and $v_2(c) = v_2(q^{y_1}+1)= 1$.  From (\ref{B3a}) we see that $v_2(q^{y_3} - 1) > 1$.  So we have 
\begin{equation} 2 \mid y_1,  2 \notmid y_2, 2 \mid y_3, 2 \notmid x_2, 2 \notmid x_3.  \label{L}  \end{equation}  
Recalling (\ref{B1}) and using (\ref{B3}) we see that consideration modulo 8 gives $p \congruent 3 \bmod 8$ so that (\ref{B2}) gives $ q \congruent 7 \bmod 8$, so that $q \ne 3$.  Now  consideration modulo 3 gives (recalling (\ref{B1}) and using (\ref{B3})) $p = 3$; also (recalling (\ref{B2})) $q \congruent 2 \bmod 3$.  To handle this case we make the following substitutions into Lemma~\ref{Lem:M} (noting $c>1$):  $a=3$, $b=q$, $x=x_3$, $y=y_3$.  We get either
\begin{equation} \frac{y_3}{ \log(3)} < 2409.08 \label{M1} \end{equation}
or 
\begin{equation}  \frac{y_3}{ \log(3)} <  \frac{2 \log(c)}{ \log(3) \log(q)} + 22.997 ( \log(y_3) - \log\log(3) + 2.405)^2. \label{M2} \end{equation}
From (\ref{B1}) and (\ref{B2}) we have $y_1 > y_2$.  From (\ref{B2}) and (\ref{B3}) we have $x_3 > x_2$.  By Lemma 12 of \cite{ScSt4} we must have 
\begin{equation} y_2 < y_1 <y_3,  \label{M2.5} \end{equation}
noting that none of the exceptional cases of Lemma 12 of \cite{ScSt4} fits Case 2. 

Combining (\ref{B2}) and (\ref{B3}) we obtain 
\begin{equation} 3^{x_2} (3^{x_3-x_2} - 1) = q^{y_2} (q^{y_3-y_2} +1).  \label{M3}  \end{equation}
If $q \congruent \pm 1 \bmod 9$ then (\ref{L}), (\ref{B1}), and (\ref{B3}) give $3^{x_3} \congruent 3 \bmod 9$ which is impossible.  So we can apply Lemma 1 of \cite{ScSt2} to (\ref{M3}) to see that 
\begin{equation} 3^{x_2 - 1} \mid y_3 - y_2. \label{M3.1} \end{equation}  
Now if $3^{x_2} < c/2$, then $q^{y_2} > c/2 > q^{y_1} / 2$, contradicting (\ref{M2.5}), so we can assume 
\begin{equation} 3^{x_2} > c/2. \label{M3.2} \end{equation} 
So now, using (\ref{M3.1}) and (\ref{M3.2}) and letting $k \ge 1$ be some real number, (\ref{M2}) becomes 
\begin{equation}  k \frac{3^{x_2 - 1}}{ \log(3)  } < \frac{ 2 ( \log(2) + x_2 \log(3) ) } { \log(3) \log(q) } + 22.997 ( \log(k) + (x_2 - 1) \log(3) - \log\log(3) + 2.405)^2.  \label{M3.5} \end{equation} 
If (\ref{M3.5}) holds for some fixed $x_2$, then it also holds for that $x_2$ taking $k=1$.  So (\ref{M3.5}), combined with (\ref{M1}), gives $x_2 \le 7$ (recalling $x_2$ odd).  Now 
\begin{equation} q^2 - q \le q^{y_1} - q^{y_2} = 3^{x_2} - 1 \le 2186,  \label{Q1} \end{equation} 
so that $  q \le 47$. 
We have already shown $q \congruent 7 \bmod 8$ and $q \congruent 2 \bmod 3$.  So $q=23$ or 47, both of which make (\ref{Q1}) impossible.  

\bigskip\bigskip

\noindent {\it Case 3}
\begin{equation}    p^{x_1} + (-1)^v  = c   \label{E1} \end{equation}
\begin{equation}    p^{x_2} + q^{y_2} = c   \label{E2} \end{equation}
\begin{equation}    q^{y_3} + c = p^{x_3},    \label{E3} \end{equation}
where $v \in \{ 0,1 \}$ and $p$ and $q$ are odd primes.  Consider first $v=1$.  Substituting (\ref{E1}) into (\ref{E2}) and (\ref{E3}) we find
$ q^{y_2} \congruent -1 \bmod p$ and $q^{y_3} \congruent 1 \bmod p$ so that $2 \mid y_3$ which, by Lemma 3, is possible only when $(p,q,c) = (3, 5, 2)$ or $(5,11,4)$, both of which cases are impossible since $c \le 4$ makes (\ref{E2}) impossible.  

Now consider $v=0$.  Substituting (\ref{E1}) into (\ref{E2}) and (\ref{E3}) we get $q^{y_2} \congruent 1 \bmod p$ and $q^{y_3} \congruent -1 \bmod p$ so that $2 \mid y_2$, which, by Lemma 3, is possible only when $(p,q,c)=(3,5,28)$ or $(5,11,126)$.  $(p,q,c) = (3,5,28)$ makes (\ref{E3}) modulo 8 incompatible with (\ref{E3}) modulo 5, while $(p,q,c) = (5, 11, 126)$ makes (\ref{E3}) modulo 8 incompatible with (\ref{E3}) modulo 3.  

\bigskip\bigskip

\noindent {\it Case 4}
\begin{equation}    2^{y_1} + (-1)^u = c   \label{F1} \end{equation}
\begin{equation}    p^{x_2} + 2^{y_2} = c   \label{F2} \end{equation}
\begin{equation}    2^{y_3} + c = p^{x_3},    \label{F3} \end{equation}
where $u \in \{0,1\}$.  From (\ref{F1}) and (\ref{F2}) we see that $y_1 \ge 3$ unless $(p,q,c)=(3,2,5)$, an excluded case.  Clearly $y_1 > y_2$ and $x_3 > x_2$, so that Lemma 12 of \cite{ScSt4} gives
\begin{equation}  y_2 < y_1 < y_3,  \label{F4}  \end{equation}
noting that the relevant exceptional cases of Lemma 12 of \cite{ScSt4} have already been excluded.    

Consider first $u=1$.  Substituting (\ref{F1}) into (\ref{F2}) and (\ref{F3}) and using (\ref{F4}), we find 
\begin{equation}  v_2(p^{x_2} + 1) = y_2 < y_1 = v_2(p^{x_3} + 1),  \label{F5} \end{equation}  
so that $p \congruent 3 \bmod 4$, $x_3$ is odd, and $x_2$ is even.  But this makes (\ref{F2}) impossible modulo 8 since $c \congruent 7 \bmod 8$ (recall $y_1 \ge 3$).  

Now consider $u=0$.  Substituting (\ref{F1}) into (\ref{F2}) and (\ref{F3}) and using (\ref{F4}), we find that 
\begin{equation*} v_2(p^{x_2} - 1) = y_2 < y_1 = v_2(p^{x_3} - 1)   \end{equation*}  
so that $2 \mid x_3$ which is impossible by Lemma 7 unless $(p,q,c) = (7,2,17)$, $(23,2,17)$, or $(2^t+1, 2, 2^{t+1}+1)$ where $t \ge 3$ (recall (\ref{I})).   The first two of these three cases make (\ref{F2}) impossible, while the third case is the already excluded $(p,q,c) = (F, 2, 2F -1)$.

\bigskip\bigskip

\noindent {\it Case 5}
\begin{equation}    2^{x_1} + (-1)^v = c   \label{G1} \end{equation}
\begin{equation}    2^{x_2} + q^{y_2} = c   \label{G2} \end{equation}
\begin{equation}    q^{y_3} + c = 2^{x_3},    \label{G3} \end{equation}
where $v \in \{0,1\}$.  
We see that $x_2 < x_1 < x_3$.  Also, $x_1 \ge 3$, otherwise (\ref{G2}) is impossible except when $(p,q,c) = (2,3, 5)$, which has been excluded.  Assume first $v=1$.  Then from (\ref{G1}) we get $c \congruent 7 \bmod 8$.  If $y_3$ is odd, then, from (\ref{G3}) we get $q \congruent 1 \bmod 8$ so that (\ref{G2}) becomes impossible modulo 8.  So $2 \mid y_3$ so that, using Lemma 6 and recalling (\ref{I}), we see from (\ref{G3}) that we must have $(p,q,c) = (2,11,7)$, $(2,181,7)$, or $(2,2^t -1, 2^{t+1} - 1)$ where $t \ge 3$.  The first two of these possibilities have $c = 7$, making (\ref{G2}) impossible, and the third possibility corresponds to the exceptional case $(2, M, 2M+1)$ which we have already excluded.  

So now assume $v=0$.  Substituting (\ref{G1}) into (\ref{G2}) and (\ref{G3}) we find that 
$$v_2(q^{y_2} -1) = x_2 < x_1 = v_2(q^{y_3} + 1), $$
which is possible only when $x_2 = 1$, so that $q=2^{x_1} -1$ and $c=2^{x_1} +1$, 
giving the exceptional case $(2, M, M+2)$, which has been excluded.

\bigskip\bigskip

\noindent {\it Case 6}
\begin{equation}    p^{x_1} + 1 = c   \label{J1} \end{equation}
\begin{equation}    q^{y_2} + c = p^{x_2}   \label{J2} \end{equation}
\begin{equation}    q^{y_3} + c = p^{x_3}    \label{J3} \end{equation}
By Theorem 4 of \cite{Sc}, $p>2$.  Substituting (\ref{J1}) into (\ref{J2}) and (\ref{J3}) we find $q^{y_2} \congruent q^{y_3} \congruent -1 \bmod p$, so that $2 \mid y_2-y_3$, contradicting Theorem 3 of \cite{Sc}.

\bigskip\bigskip

\noindent {\it Case 7}
\begin{equation}    q^{y_1} + 1 = c   \label{K1} \end{equation}
\begin{equation}    q^{y_2} + c = p^{x_2}   \label{K2} \end{equation}
\begin{equation}    q^{y_3} + c = p^{x_3}    \label{K3} \end{equation}
By Theorems 3 and 4 of \cite{Sc}, $p>2$ and $2 \notmid y_2-y_3$.  If $2 \mid x_2-x_3$, then $p^{x_2} \congruent p^{x_3} \bmod 3$ and $p^{x_2} \congruent p^{x_3} \bmod 4$, so that
$$ q^{y_2} = p^{x_2} - c \congruent p^{x_3} - c = q^{y_3} \bmod 12, $$
so that $q \congruent 1 \bmod 12$, $c \congruent 2 \bmod 12$, and (\ref{K2}) gives $p = 3$, contradicting Corollary 1.7 of \cite{Be}.  

So we must have $2 \notmid x_2-x_3$.  Without loss of generality take $x_2$ even and $x_3$ odd.  
Assume first $q>2$.  
Then from (\ref{K2}) we see that $q^{y_2} + q^{y_1} + 1 $ is a square, impossible by Lemma 5.  
So $q=2$, and we can use equations (2), (4), and (6a) of \cite{ScSt} to see that $2^{y_2} \mid\mid p-1$.  Now rewrite (\ref{K2}) as 
\begin{equation*}  2^{y_2} + (c-1) = (p^{x_2} - 1) \end{equation*} 
to see that we must have $y_1 = y_2$, making the left side of (\ref{K2}) less than $2p$, which is impossible.

\bigskip\bigskip

\noindent {\it Case 8}
\begin{equation}    p^{x_1} + 1 = c   \label{N1} \end{equation}
\begin{equation}    q^{y_2} + c = p^{x_2}   \label{N2} \end{equation}
\begin{equation}    p^{x_3} + c = q^{y_3}    \label{N3} \end{equation}
Assume first $p>2$.  Substituting (\ref{N1}) into (\ref{N2}) and (\ref{N3}) we find $q^{y_2} \congruent -1 \bmod p$ and $q^{y_3} \congruent 1 \bmod p$, so that $2 \mid y_3$, contradicting Lemma 5.

So $p=2$.  Assume first $x_1=1$ so that $c=3$.  Then $q^{y_2} \congruent 5 \bmod 8$, so that considering (\ref{N3}) modulo 8 we get $2 \notmid y_3$, $x_3 = 1$, $(p,q,c) = (2,5,3)$, an excluded case.  Assume next $x_1=2$ so that $c=5$.  If $q=3$, we have the excluded case $(p,q,c)=(2,3,5)$, so we can assume $q>3$.  Considering (\ref{N2}) and (\ref{N3}) modulo 3 we get $q^{y_2} \congruent 2$, $q^{y_3} \congruent 1 \bmod 3$, $2 \mid y_3$, $q^{y_3} \congruent 1 \bmod 8$, $x_3 = 2$, $q=3$, a contradiction.   So $x_1>2$.  (\ref{N2}) requires $q \congruent 7 \bmod 8$ with $\Legendre{c}{q} = 1$; but then (\ref{N3}) gives $\Legendre{c}{q} = -1$, a contradiction.

\bigskip\bigskip

\noindent {\it Case 9}
\begin{equation}    p^{x_1} + (-1)^w = c   \label{R1} \end{equation}
\begin{equation}    p^{x_2} + q^{y_2} = c   \label{R2} \end{equation}
\begin{equation}    p^{x_3} + q^{y_3} = c,    \label{R3} \end{equation}
where $w \in \{0,1\}$.  This case can be handled using essentially the same method as used to handle the case (31) in Theorem 7 of \cite{ScSt}.

\bigskip\bigskip

\noindent {\it Case 10}
\begin{equation}    1 + 1 = 2   \label{S1} \end{equation}
\begin{equation}    q^{y_2} + 2 = p^{x_2}   \label{S2} \end{equation}
\begin{equation}    q^{y_3} + 2 = p^{x_3}    \label{S3} \end{equation}
By Theorem 6 of \cite{ScSt} we cannot have both (\ref{S2}) and (\ref{S3}).

\bigskip\bigskip

\noindent {\it Case 11}
\begin{equation}    1+1=2   \label{T1} \end{equation}
\begin{equation}    q^{y_2} + 2 = p^{x_2}   \label{T2} \end{equation}
\begin{equation}    p^{x_3} + 2 = q^{y_3}    \label{T3} \end{equation}
First suppose $p \congruent q \congruent 7 \bmod 8$.  Then (\ref{T2}) and (\ref{T3}) give $\Legendre{q}{p} = \Legendre{p}{q} = -1$, impossible when $p \congruent q \congruent 3 \bmod 4$.  

Now consideration modulo 8 with consideration modulo 3 shows that one of (\ref{T2}) or (\ref{T3}) must be of the form $x^2 + 2 = 3^{n}$ for some integers $x>1$ and $n>1$; by Lemma 2 the only possibility is $(p,q,c) = (3,5,2)$ or $(5,3,2)$ which has been excluded.

\bigskip\bigskip 

Now we consider cases of three or more solutions to (\ref{1c}) with at least two solutions in which $\min(x,y)=0$.  Clearly there are at most two solutions with $\min(x,y)=0$.  Take $\delta \in \{ 0,1\}$.  If $\min(p,q)=2$, then $c$ is odd so that the only possibility allowing two solutions with $\min(x,y)=0$ is $c=3$, and we have 
\begin{equation}  2+1 = 3, 2^2 - 1 = 3, - 2^{x_3} + q^{y_3} = (-1)^\delta 3. \label{UV1} \end{equation}
If $c=2$ we have the possibility of the following three solutions:  
\begin{equation} 1+1=2, 3-1=2, -3^{x_3} + q^{y_3} = (-1)^\delta 2. \label{UV2} \end{equation} 
If $\delta = 0$ in either (\ref{UV1}) or (\ref{UV2}) then $y_3 = 1$, by Lemma 2 of \cite{ScSt}.  

Now assume $\delta = 1$.  In (\ref{UV1}) $x_3 > 2$ and consideration modulo 8 gives $y_3$ odd.  So taking $w = z = 1$, we have
$$ (-q)^{y_3} + 2^{x_3} w^{y_3} = 3 z^2, $$
from which we find that $y_3$ has no prime factor greater than or equal to 7 by Theorem 1.2 of \cite{BS}.  Assume $y_3 > 1$ and recall $y_3$ odd in (\ref{UV1}).  Then taking $g \in \{3,5\}$, we are left with the Thue equations
$$ x^g - 2^k y^g = -3, $$
where $0 < k < g$ is chosen so that $x_3 \congruent k \bmod g$ (the case $k=0$ is clearly impossible); the solutions to these Thue equations can be found using the PARI/GP command thue (see  \cite{Pari}), yielding only the single relevant case $(p,q,c) = (2,5,3)$, which has been excluded.  

If $y_3 > 1$ in (\ref{UV2}), then again $\delta = 1$ and Lemma 2 of this paper shows that $y_3$ is odd (recall (\ref{I})), so that, taking $w=z=1$ we have
$$ (-q)^{y_3} + 3^{x_3} w^{y_3} = 2 z^3, $$
from which we find that $y_3$ has no prime factor greater than 3 by Theorem 1.5 of \cite{BVY}.  So $3 \mid y_3$, so that, considering (\ref{UV2}) modulo 9, we get $x_3 = 1$, impossible.  

So $y_3 = 1$ in both (\ref{UV1}) and (\ref{UV2}), and we obtain the last two exceptions in the formulation of Theorem 3.  

Assume neither (\ref{UV1}) nor (\ref{UV2}) holds.  Then, in considering cases of three or more solutions to (\ref{1c}) with at least two solutions in which $\min(x,y) = 0$, 
we can assume that $\min(p,q) > 2$ and also that no solution has $x=y=0$.  Thus it remains to consider  
\begin{equation*}    p^{x_1} =  c +(-1)^w   \end{equation*}
\begin{equation*}    q^{y_2} = c - (-1)^w   \end{equation*}
\begin{equation*}    (-1)^u p^{x_3} + (-1)^v q^{y_3} = c     \end{equation*}
where $\min(x_1, y_2, x_3, y_3) > 0$, $u,v,w \in \{0,1\}$, and $\min(p,q) > 2$.  If $(u,v) = (0,0)$, then 
$$\frac{c + (-1)^w}{p} + \frac{c-(-1)^w}{q} = p^{x_1 - 1} + q^{y_2 - 1} \ge p^{x_3} + q^{y_3} = c,$$
impossible when $\min(p,q)>2$.  So it suffices to consider only the two cases given below by (\ref{C1}), (\ref{C2}), (\ref{C3}), and (\ref{D1}), (\ref{D2}), (\ref{D3}).

\bigskip\bigskip

\noindent {\it Case 12}
\begin{equation}    p^{x_1} + 1  = c   \label{C1} \end{equation}
\begin{equation}    1+ c = q^{y_2}   \label{C2} \end{equation}
\begin{equation}    q^{y_3} + c = p^{x_3}    \label{C3} \end{equation}
where $p$ and $q$ are odd primes. 
 
From (\ref{C1}) and (\ref{C2}) we have
 \begin{equation}  \Legendre{c}{p}  = 1  \label{C4} \end{equation}  
and 
\begin{equation} \Legendre{c}{q} = \Legendre{-1}{q}. \label{C5} \end{equation}
From (\ref{C1}) and (\ref{C2}) we see that $p$ and $q$ cannot both be congruent to $1 \bmod 4$.  Considering the remaining possibilities for $p$ and $q$ modulo 4, we see that (\ref{C4}) and (\ref{C5}) are incompatible with (\ref{C3}) when $2 \notmid x_3 y_3$.  And substituting (\ref{C1}) into (\ref{C3}) and applying Lemma 4, we see that $x_3$ and $y_3$ cannot both be even.  So $2 \notmid x_3-y_3$.  Assume $2 \mid y_3$.  
Then combining (\ref{C4}) and (\ref{C3}) we see that $p \congruent 1 \bmod 4$, so that (\ref{C1}) gives $c \congruent 2 \bmod 4$   while (\ref{C3}) gives $c \congruent 0 \bmod 4$.   So we are left with $2 \mid x_3$ and $2 \notmid y_3$.  From (\ref{C5}) and (\ref{C3}) we now obtain $q \congruent 1 \bmod 4$, so that $c \congruent 0 \bmod 4$ and, from (\ref{C1}), $p \congruent 3 \bmod 4$ with $x_1$ odd.  If $2 \mid y_2$, then, since $2 \notmid x_1$, $2 \mid x_3$, and $2 \notmid y_3$, we have 
$$ v_2(c) = v_2(q^{y_2} - 1) > v_2(q^{y_3} - 1) =v_2(p^{x_3} - 1) > v_2(p^{x_1} +1) = v_2(c),$$  
a contradiction.  So we have
\begin{equation}  2 \notmid x_1, 2 \notmid y_2, 2 \mid x_3, 2 \notmid y_3.    \label{C6}  \end{equation}  
If $3 \notmid pq$, then $3 \mid c$ and $p \congruent 2 \bmod 3$.  So now we have $p \congruent 11 \bmod 12$ so that
$$ \frac{p-1}{2} \congruent 5 \bmod 6 $$
and there must be an odd prime $r$ dividing $p-1$ such that $r \congruent 2 \bmod 3$.  We have $p^{x_1} \congruent p^{x_3} \congruent 1 \bmod r$, $c \congruent 2 \bmod r$, $q^{y_2} \congruent 3 \mod r$, $q^{y_3} \congruent -1 \bmod r$.  But since $2 \mid y_3-y_2$, we must have
$$ \Legendre{3}{r} = \Legendre{-1}{r},   $$
which is impossible when $r \congruent 2 \bmod 3$.  

So $3 \mid pq$ and, recalling $q \congruent 1 \bmod 4$, we are left with $p=3$.  We recall (\ref{C6}) and consider (\ref{C1}), (\ref{C2}), and (\ref{C3}) modulo 5.  $p^{x_1} \congruent \pm 2 \bmod 5$.  If $p^{x_1} \congruent 3 \bmod 5$ then, using (\ref{C1}) and (\ref{C2}), we get $3^{x_1} + 2 = 5^{y_2}$ so that Theorem 3 of \cite{Sc} gives $x_1 = y_2 = 1$, $c=4$, which has been excluded.  So $p^{x_1} \congruent 2 \bmod 5$, $c \congruent 3 \bmod 5$, $q^{y_2} \congruent q^{y_3} \congruent 4 \bmod 5$, so that (\ref{C3}) requires $p^{x_3} \congruent 2 \bmod 5$, contradicting $2 \mid x_3$  as in (\ref{C6}).

\bigskip\bigskip

\noindent {\it Case 13}
\begin{equation}    1+c = p^{x_1}   \label{D1} \end{equation}
\begin{equation}    q^{y_2} + 1 = c   \label{D2} \end{equation}
\begin{equation}    q^{y_3} + c = p^{x_3}    \label{D3} \end{equation}
where $p$ and $q$ are odd primes.  

Substituting (\ref{D1}) into (\ref{D3}) and applying Lemma 3 we find that we can assume $y_3$ is odd, since the exceptional cases of Lemma 3 make (\ref{D2}) impossible since $c \le 4$ and $q \ge 5$.  Substituting (\ref{D2}) into (\ref{D3}) and applying Lemma 5, we find that we can assume $x_3$ is odd.  So 
\begin{equation} 2 \notmid x_3, 2 \notmid y_3. \label{D3.5} \end{equation}
We have
\begin{equation}  \Legendre{c}{q} = 1 \label{D4} \end{equation}  
and
\begin{equation} \Legendre{c}{p} = \Legendre{-1}{p}. \label{D5} \end{equation} 
If $2 \mid x_1$, we have $4 \mid c$, $q \congruent 3 \bmod 4$, and, from (\ref{D3}) and (\ref{D3.5}), $p \congruent 3 \bmod 4$.  Combining (\ref{D4}) with (\ref{D3}) we get $\Legendre{p}{q} = 1$, while combining (\ref{D5}) with (\ref{D3}) we get $\Legendre{q}{p}=1$, which is impossible when $p \congruent q \congruent 3 \bmod 4$.  So $2 \notmid x_1$.  

Therefore, if $3 \mid c$, (\ref{D1}) gives $p \congruent 1 \bmod 3$.  But (\ref{D2}) gives $q  \congruent 2 \bmod 3$, and, from (\ref{D3}) and (\ref{D3.5}), we have a contradiction.  So $3 \notmid c$.  

So $3 \mid pq$.  If $q=3$ then, from (\ref{D2}), we get $c \congruent 1 \bmod 3$, and, from (\ref{D1}) we get $p \congruent 2 \bmod 3$.  But then (\ref{D3}) requires $2 \mid x_3$, contradicting (\ref{D3.5}).  So $p=3$.  

To handle the case $p=3$, we use Lemma~\ref{Lem:M} with the following substitutions:  $a=3$, $b=q$, $x=x_3$, $y=y_3$.  Then by Lemma~\ref{Lem:M} (noting $c > 1$) we must have either (\ref{M1}) or (\ref{M2}).  Combining (\ref{D1}) and (\ref{D3}) we get 
$$ 3^{x_1} (3^{x_3-x_1} - 1) = q^{y_3} - 1  $$
so that 
\begin{equation}  3^{x_1} \mid q^{y_3} - 1.  \label{D6} \end{equation}
From (\ref{D1}) and (\ref{D2}) we get  $x_1 > 1$, so that $q^{y_2} \congruent 7 \bmod 9$, $q \not\congruent \pm 1 \bmod 9$.  Applying Lemma 1 of \cite{ScSt2} to (\ref{D6}),  we have
\begin{equation}  3^{x_1 - 1} \mid y_3. \label{D7} \end{equation} 
Using (\ref{D7}) and (\ref{D1}) and noting that if (\ref{2m}) holds for $G = G_1 > G_0 > 1$ it holds for $G = G_0$, we see that (\ref{M1}) and (\ref{M2}) can be replaced by  
\begin{equation}  \frac{3^{x_1-1} }{\log(3)} < 2409.08 \label{D8} \end{equation} 
and 
\begin{equation} \frac{3^{x_1-1} }{\log(3)} < \frac{ 2 x_1 }{\log(q)} + 22.997 ( (x_1-1) \log 3 - \log\log 3 + 2.405)^2,      \label{D9} \end{equation}  
giving $x_1 \le 8$.  Using (\ref{D1}), (\ref{D2}), (\ref{D3}), and (\ref{D3.5}) we have
\begin{equation} 3^{x_1} -2 = q^{y_2}, 3^{x_3} \congruent 1 \bmod q, 2 \notmid x_3. \label{D10} \end{equation}
We easily check that (\ref{D10}) is impossible for $x_1 =3$, 5, or 7 (recall $2 \notmid x_1 > 1$).

\bigskip\bigskip

We have now shown that the list of exceptional cases in Theorem 3 includes all $(p,q,c)$ allowing at least three solutions to (\ref{1c}).  It remains to show that for each such $(p,q,c)$ the list of solutions $(x,y)$ is complete.  

Consider first $(p,q,c) = (2, 2^t + (-1)^\delta 3, 3)$ which gives the three solutions 
$$ 2+1 = 2^2 - 1 = - (-1)^\delta 2^{x_3} + (-1)^\delta q^{y_3} =3,$$
where $y_3=1$ and $x_3 = t>1$.  If $q = 2^t + 3$, then we cannot have $q^{y_4} + 3 = 2^{x_4}$ since $q \congruent 3 \bmod 4$.  So any further solution $(x_4, y_4)$ must be of the form $2^{x_4} + 3 = q^{y_4}$ with $y_4$ odd so that $q^{y_4} \congruent q^{y_3} \bmod 3$, giving $2 \mid x_3 - x_4$, contradicting Theorem 3 of \cite{Sc}, so that there exactly three solutions in this case.  Similarly, the case $q = 2^t - 3$ gives exactly three solutions (note that $t$ is defined so that $q \ne 5$).  

Now consider $(p,q,c) = (3, 3^n +(-1)^\delta 2, 2)$ which gives the three solutions
$$ 1 + 1 = 3 - 1 = -(-1)^\delta 3^{x_3} + (-1)^\delta q^{y_3} = 2. $$
By the results given in Cases 10 and 11, this case also has exactly three solutions (except for the excluded case $(3,5,2)$).  

The remaining cases can be handled either by Theorem 2 of \cite{Sc} or by Observation 8 of \cite{ScSt2}.  
\end{proof}

The proof of Theorem 3 is elementary except for the use of Lemma~\ref{Lem:M} (to handle the case $p=3$ in Cases 2 and 13), Corollary 1.7 of \cite{Be} (to handle the case $p=3$ in Case 7), Lemma 7 (Case 4), Lemma 12 of \cite{ScSt4} (Cases 2 and 4), Observation 8 of \cite{ScSt2} (at the end of the proof of Theorem 3), and, finally, Theorem 1.2 of \cite{BS}, Theorem 1.5 of \cite{BVY}, and Pari (to obtain $y_3=1$ in (\ref{UV1}) and (\ref{UV2})).  The following Lemma~\ref{Lem:MF} allows us to replace Observation 8 by an elementary result, and the Corollary to Lemma~\ref{Lem:MF} shows that Lemma 12 of \cite{ScSt4} can be given an elementary proof; also the somewhat longer alternate proof of Case 4 of Theorem 3 given below removes the dependence on Lemma 7, thus removing the dependence on \cite{BB}.  Finally, rewriting the last two exceptional $(p,q,c)$ in the formulation of Theorem 3 as $(3, (3^n + (-1)^\delta 2)^{1/m}, 2)$ and $(2, (2^t +(-1)^\delta 3)^{1/m}, 3)$ where $m \ge 1$ is an integer, we can remove the need for Theorem 1.2 of \cite{BS}, Theorem 1.5 of \cite{BVY}, and Pari.  With these changes the proof of Theorem 3 is lengthened but becomes elementary except for three applications (all with $\min(p,q)=3$) of lower bounds on linear forms in logarithms (note that Corollary 1.7 of \cite{Be} and Lemma~\ref{Lem:M} both use a theorem of Mignotte \cite{Mi} as used in \cite{Be}).

\begin{Lemma}  
\label{Lem:MF}

If $(p,q,c) = (2, M, M+2)$ where $M = 2^t -1 >3$ is a Mersenne prime, the only solutions to (\ref{1c}) are
\begin{equation}  2^t + 1 = c, \label{MFA1} \end{equation} 
\begin{equation}  2 + M = c,  \label{MFA2}  \end{equation} 
\begin{equation}  2^{t+1} - M = c.  \label{MFA3}  \end{equation} 

If $(p,q,c) = (2, M, 2M + 1)$ where $M = 2^t -1 >3$ is a Mersenne prime, the only solutions to (\ref{1c}) are
\begin{equation}  2^{t+1} - 1 = c  \label{MFB1}  \end{equation} 
\begin{equation}  2^{t} + M = c  \label{MFB2} \end{equation} 
\begin{equation}  2^{2t} - M^2 = c  \label{MFB3}  \end{equation} 

If $(p,q,c) = (2, F, F-2)$ where $F = 2^t +1 >5$ is a Fermat prime, the only solutions to (\ref{1c}) are
\begin{equation}  2^t - 1 = c  \label{MFC1}  \end{equation} 
\begin{equation}  -2 + F = c  \label{MFC2}   \end{equation} 
\begin{equation}  2^{t+1} - F = c  \label{MFC3} \end{equation}  

If $(p,q,c) = (2, F, 2F-1)$ where $F = 2^t +1 >5$ is a Fermat prime, the only solutions to (\ref{1c}) are
\begin{equation}  2^{t+1} + 1 = c  \label{MFD1}  \end{equation} 
\begin{equation}  2^t + F = c  \label{MFD2}   \end{equation} 
\begin{equation}  -2^{2t} + F^2 = c \label{MFD3} \end{equation} 

\end{Lemma}

\begin{proof} 
Let $M = 2^t - 1 > 3$ be a Mersenne prime and let $c$ be either $2^t + 1$ or $2^{t+1} - 1$.  Then $\Legendre{c}{M} = 1$ and the equation $2^x + c = M^y$ is impossible.  Considering the equation $M^y + c = 2^x$ modulo 8, we see that the parity of $y$ is determined, so, by Theorem 3 of \cite{Sc}, the only solutions to this equation with $y>0$ are given by (\ref{MFA3}) and (\ref{MFB3}) respectively.  Further, it is easily seen that the only cases of solutions to the equation $2^x + M^y = c$ with $y>0$ are given by (\ref{MFA2}) and (\ref{MFB2}) respectively.  And clearly the only solutions with $\min(x,y) = 0$ are given by (\ref{MFA1}) and (\ref{MFB1}) respectively.  

Now let $(p,q,c) = (2, F, F-2)$ where $F = 2^t +1 >5$ is a Fermat prime.  Consideration modulo 8 shows that the equation $2^x + c = F^y$ requires $x=1$ giving (\ref{MFC2}).  Consideration modulo $2^{t+1}$ shows that the equation $F^y + c = 2^x$ requires $y$ odd when $y>0$, so, by Theorem 3 of \cite{Sc}, we  must have (\ref{MFC3}).  Clearly there can be no solutions to the equation $2^x + F^y = c$ with $y>0$.  Finally, the only solution for this $(p,q,c)$ with $\min(x,y)=0$ is given by (\ref{MFC1}).    

Now let $(p,q,c) = (2, F, 2F-1)$ where $F = 2^t +1 >5$ is a Fermat prime.  Consideration modulo 3 shows that the equation $2^x + c = F^y$ requires $2 \mid x-y$; when $x$ and $y$ are odd, consideration modulo $2^{t+1}$ shows that we must have $x=t$, which is impossible since $F < 2^t + c < F^2$, and, if $x$ and $y$ are even, the only solution is given by (\ref{MFD3}) by Theorem 3 of \cite{Sc}.  
Consideration modulo 8 shows that the equation $F^y + c = 2^x$ is impossible.  Clearly the only solution to the equation $2^x + F^y = c$ with $y>0$ is given by (\ref{MFD2}).  And it is also clear the only possible solution with $\min(x,y)=0$ is given by (\ref{MFD1}).
\end{proof}

\begin{Corollary}
Lemma 12 of \cite{ScSt4} has an elementary proof.  
\end{Corollary}

\begin{proof} 
The proof of Lemma 12 of \cite{ScSt4} depends only on the lemmas preceding it in that paper, which in turn are elementary except for use of Theorems 1 and 7 of \cite{ScSt2}.  But in every case the use of Theorems 1 and 7 of \cite{ScSt2} can be replaced by the use of either Theorem 2 of \cite{Sc} or Lemma~\ref{Lem:MF} above.  
\end{proof}

The following eliminates the dependence of Case 4 of Theorem 3 on Lemma 7.  
 
\begin{proof}[Alternate Proof of Case 4 of Theorem 3]
It suffices to treat only the case $u=0$, noting $y_1 \ge 3$ and recalling $2 \mid x_3$.  If $p \congruent 7 \bmod 8$ then (\ref{F2}) requires $\Legendre{c}{p} = 1$ while (\ref{F3}) requires $\Legendre{c}{p} = -1$, so
\begin{equation} p \not\congruent 7 \bmod 8.  \label{Alt0} \end{equation}

If $y_2 = 1$ then $p^{x_2} = 2^{y_1} - 1 \congruent 7 \bmod 8$, impossible by (\ref{Alt0}), so $p^{x_2} \congruent 1 \bmod 4$.  If $2 \mid x_2$ then, using Lemma 6 with (\ref{I}) and (\ref{Alt0}), we must have $(p, c) = (11, 129)$ or $(181, 32769)$, so considering (\ref{F3}) modulo 5 we find $2 \notmid y_3$, while considering (\ref{F3}) modulo 3 we find $2 \mid y_3$ since $2 \mid x_3$.  So 
\begin{equation}  2 \notmid x_2,   \label{Alt1} \end{equation}  
and
\begin{equation}  p \congruent 1 \bmod 4.  \label{Alt1b} \end{equation} 
Assume now $4 \mid x_3$ and recall (\ref{I}).  Then, since 
\begin{equation} 2^{y_3} + 2^{y_1} + 1 = p^{x_3}, \label{Alt2} \end{equation}
consideration modulo 5 gives $2^{y_3} + 2^{y_1} \congruent 0 \bmod 5$, so that $2 \mid y_3 - y_1$.  But consideration of (\ref{Alt2}) modulo 3 gives $2 \notmid y_3-y_1$, a contradiction.  So
\begin{equation}   2 \mid\mid x_3.  \label{Alt3} \end{equation}

Let $k = v_2(p-1)$.  Then, using (\ref{Alt1}) and (\ref{Alt3}), we have
\begin{equation}  v_2(p^{x_2} - 1) = v_2(p^{ x_3/2} -1) = k. \label{Alt4} \end{equation} 
From (\ref{Alt4}) and (\ref{Alt1b}) we have $v_2(p^{x_3} - 1) = k+1$, so from (\ref{Alt2}) and (\ref{F4}) we have $y_1 = k+1$, so from (\ref{F2}) and (\ref{Alt4}) we have $y_2 = k$, $p=2^k + 1$ (note $x_2=1$ by (\ref{I})), giving the already excluded case $(p,q,c) = (F, 2, 2F-1)$.  
\end{proof}

\section{Further Related Results}  

In this section we show how Lemma~\ref{L2} can be used in a different direction, treating an old problem which has already received much attention (see Introduction).

\begin{Theorem} 
\label{Thm4}
Let $C$ be an even positive integer, and let $PQ$ be the largest squarefree divisor of $C$, where $P$ is chosen so that $(C/P)^{1/2}$ is an integer.  If the equation   
\begin{equation}x^2 + C = y^n \label{(4.1)}\end{equation}  
has a solution $(x,y,n)$ with $x$ and $y$ nonzero integers divisible by at most one prime, $(x,y)=1$, $n$ a positive integer, and $(x,y,n) \ne (7,3,4)$ or $(401,11,5)$, then we must have either $n=3$ or   
$$n | N = 2 \cdot 3^u h(-P) \langle q_1 - \Legendre{-P}{q_1}, \dots, q_n - \Legendre{-P}{q_n} \rangle$$  
Here $u=1$ or 0 according as $3 < P \congruent 3 \bmod 8$ or not, $h(-P)$ is the lowest $h$ such that $\scripta^h$ is principal for every ideal $\scripta $ in $\ratQ(\sqrt{-P})$, $\langle a_1, a_2, \dots, a_n \rangle$ is the least common multiple of the members of the set $S = \{ a_1, a_2, \dots, a_n \}$ when $S \ne \emptyset$, $\langle a_1, a_2, \dots, a_n \rangle = 1$ when $S = \emptyset$, $q_1 q_2 \dots q_n = Q$ is the prime factorization of $Q$, and $\Legendre{a}{q}$ is the familiar Legendre symbol unless $q=2$ in which case $\Legendre{a}{2} = 0$. 
\end{Theorem} 
  
\begin{proof}
  It suffices to prove the theorem for the case in which $y$ is a positive prime.  Assume there exists a solution to (\ref{(4.1)}).  Let $\scriptp \bar{\scriptp}$ be the prime ideal factorization of $y$ in $\ratQ(\sqrt{-P})$.  Let $k$ be the smallest number such that $\scriptp^k = [\alpha]$ is principal with a generator $\alpha$ having integer coefficients.  When $P=1$, we choose $\alpha$ so that the coefficient of its imaginary term is even.  When $P = 3$ we can take $k=1$. Then  
$$\alpha^{n/k} = \pm x \pm \sqrt{-C}$$  
where the $\pm$ signs are independent.  Note that when $P=3$ and $\alpha^{n/k} \epsilon = x \pm \sqrt{-C} $ for some unit $\epsilon$, we must have $\epsilon = \pm 1$.  Let $j$ be the least number such that $\alpha^j = u + v Q \sqrt{-P}$ for some integers $u$ and $v$.  By elementary properties of the coefficients of powers of integers in a quadratic field, $jk | N/2$.  Also, $jk | n = jkr$ for some $r$.    
So we have  
$$(u + v Q \sqrt{-P})^r = \pm x \pm \sqrt{-C} $$  
If $r=1$ or $r=2$, Theorem~\ref{Thm4} holds, so assume $r \ge 3$.    
  
If $r$ is even, then any prime dividing $u$ must divide $C$, since $\pm x \pm \sqrt{-C}$ must be divisible by $( u  + v Q \sqrt{-P})^2.$  Since $(u, C) = 1$, we must have $u = \pm 1$ when $r$ is even.    
  
If $r$ is odd, then $u$ divides $x$.  $x = \pm 1$ implies $u = \pm 1$.  Assume $| x | > 1$.  Let $x = \pm g^s$ where $g$ is a positive prime and $s > 0$.  Then, when $r$ is odd, $u = \pm g^t$ for some $t \ge 0$.  Also, every prime dividing $v$ divides $C$.  Thus, if $t>0$, then by Theorem 1 of \cite{Sc}, $r=1$ which we already excluded.   (Note that the only relevant exceptional case in Theorem 1 of \cite{Sc} is $(x,y,C) = (3,13,10)$, in which case $n = 1$ or $3$.)   
  
So $u = \pm 1$ regardless of the value of $x$ or the parity of $r$.  Letting $D = v^2 Q^2 P$, we have  
$$(1+ \sqrt{-D})^r = \pm x \pm w \sqrt{-D} $$  
for some positive integer $w$.  If $w=1$, we see from Lemma~\ref{L2} that $r=3$ and $j = k = 1$, so that $n = 3$ and the theorem holds.    
  
So $w>1$, and $w$ is divisible only by primes dividing $C$.  In what follows, we apply Lemmas 1--3 of \cite{Sc}.  We must have at least one prime $r_1$ dividing $C$ which also divides $r$.  We have, for any such $r_1$,  
\begin{equation}(1+\sqrt{-D} )^{r_1} = \pm x_1 \pm w_1 \sqrt{-D} \label{(4.2)}\end{equation}  
where $w_1 | w$.  If $r_1$ is odd, we have  
\begin{equation} \pm w_1= r_1 - {r_1 \choose 3} D + {r_1 \choose 5} D^2 - \dots \pm D^{{r_1 -1 \over 2}}. \label{(4.3)}\end{equation}   
$r_1 | w_1$, and, if $r_1 > 3$, then $r_1^2 \notmid w_1$.  Also, when $r_1 > 3$, $(w_1 / r_1, C) =1$, so that $w_1 = \pm r_1$.    
  
If $r_1 = 3$, we must have $w_1 = \pm 3^z$ for some $z > 0$ so that $D= 3^z + 3$.  Now $1+D$ is the norm of $\alpha^{j}$ which equals $y^{jk}$.  But $1+D = 3^z + 4$ cannot be a perfect power of $y$ by Lemma 2 of \cite{ScSt}.  So $j=k=1$.  Now $| x_1 | = 3D-1 > 1$.  Also, $(x_1, C) = 1$ so $2  \notmid \frac{r}{r_1}$.  Thus, $x_1$ must be a power of the prime dividing $x$ (this follows from the same kind of elementary reasoning used for Lemmas 1--3 of \cite{Sc}).  By Theorem 1 of \cite{Sc}, $r= r_1$, $n = 3jk = 3$, and the theorem holds.    
  
If $r_1 = 5$ then (\ref{(4.3)}) shows that $\pm 5 = 5 - 10 D + D^2$.  Since $5 | D$, this implies $D=10$, $y^{jk}=11$ which gives $(x_1, y, r_1, j,k) = (401, 11, 5, 1, 1)$.  If $r> r_1$, we must have $2 \notmid r$ and $401 | x$, so Theorem 1 of \cite{Sc} shows $r=r_1$. This leads to the case $(x,y,n) = (401, 11, 5)$.    
  
If $r_1 \ge 7$, (\ref{(4.3)}) is impossible for $w_1 = \pm r_1$.   
  
Finally, it remains to consider $r = 2^h$, $h > 1$.  Then we have (\ref{(4.2)}) with $r_1 = 2$, $| x_1 | = D-1$.  If $D > 2$, then, since $D-1>1$, we have $2 \notmid \frac{r}{r_1}$, contradicting $h>1$.     
So $D=2$, so that $y^{jk} = 1+ D = 3$, and $n = r = 2^h$.   $n = 4$ gives the exceptional case $(x,y,n) = (7,3,4)$; and $n>4$ gives $7 \mid w$, impossible.    
\end{proof}

\bigskip

Author addresses:  

Reese Scott, 86 Boston Street, Somerville, MA 02143

Robert Styer, Villanova University, Department of Mathematics and Statistics, 800 Lancaster Avenue, Villanova, PA  19085 \quad  robert.styer@villanova.edu

\end{document}